\documentclass[12pt, leqno]{article}
\usepackage{amsmath,amsfonts,amssymb,wasysym,mathrsfs}
\usepackage[latin1]{inputenc}
\usepackage{shapepar}
\usepackage{graphicx}
\usepackage{cancel}
\usepackage{float}
\usepackage{hyperref}
\usepackage{ifpdf} 
\usepackage{color}
\usepackage{amsbsy}
\newcommand{\re}{{\mathbb R}}
\newcommand{\R}{{\mathbb R}}
\newcommand{\ren}{{\mathbb R}^N}

\newcommand{\be}[1]{\begin{equation}\label{#1}}
\newcommand{\ee}{\end{equation}}

\newcommand{\prf}{\par\smallskip\noindent{\sl Proof. \/}}
\newcommand{\finprf}{\unskip\null\hfill$\;\square$\vskip 0.3cm}

\newenvironment{proof}{\prf}{\finprf}
\newtheorem{theorem}{Theorem}[section]
\newtheorem{lemma}{Lemma}[section]

\newtheorem{proposition}{Proposition}[section]
\newtheorem{remark}{Remark}[section]
\newtheorem{definition}{Definition}[section]

\newcommand{\ve}{\varepsilon}

\numberwithin{equation}{section}

\def\qed{\,\unskip\kern 6pt \penalty 500
\raise -2pt\hbox{\vrule \vbox to8pt{\hrule width 6pt
\vfill\hrule}\vrule}\par}
\definecolor{darkblue}{rgb}{0.05, .05, .65}
\definecolor{darkgreen}{rgb}{0.1, .65, .1}
\definecolor{darkred}{rgb}{0.8,0,0}
\newcommand{\blue}{\color{blue}}

\begin{document}
\title{\textbf{ Anisotropic Fast Diffusion Equations}\\[7mm]}

\author{\Large  Filomena Feo\footnote{Dipartimento di Ingegneria, Universit\`{a} degli Studi di Napoli
\textquotedblleft Parthenope\textquotedblright, Centro Direzionale Isola C4
80143 Napoli, \    E-mail: {\tt filomena.feo@uniparthenope.it}}  \,, \quad
\Large Juan Luis V\'azquez \footnote{Departamento de Matem\'aticas, Universidad Aut\'onoma de Madrid, 28049 Madrid, Spain. \newline
E-mail:
{\tt juanluis.vazquez@uam.es}} \,,
\\[8pt] \Large  and
 \ Bruno Volzone \footnote{Dipartimento di Scienze e Tecnologie, Universit\`{a} degli Studi di Napoli ''Parthenope", Centro Direzionale Isola C4, 80143 Napoli, Italy. \    E-mail: {\tt bruno.volzone@uniparthenope.it}}}

\date{\today} 

\maketitle

\begin{abstract}
We prove the existence of self-similar fundamental solutions (SSF) of the anisotropic porous medium equation in the suitable fast diffusion range. Each of such SSF solutions is uniquely determined by its mass. We also obtain the asymptotic behaviour of all finite mass solutions in terms of the family of  self-similar fundamental solutions.  Time decay   rates  are derived as well as other  properties of the solutions,    like quantitative boundedness, positivity and regularity. The combination of self-similarity and anisotropy is essential in our analysis and creates serious mathematical difficulties that are addressed by means of novel methods.

\end{abstract}

\setcounter{page}{1}


\


\

\numberwithin{equation}{section}

\section{Introduction }
This paper focusses on the study of the existence and uniqueness of  self-similar fundamental solutions  to the following \emph{anisotropic porous medium equation} (APME)
\begin{equation}\label{APM}
u_t=\sum_{i=1}^N(u^{m_i})_{x_i x_i}\quad \mbox{in  } \ \quad Q:=\mathbb{R}^N\times(0,+\infty)
\end{equation}
with $N\geq2$ and $m_i>0$ for $i=1,...,N$.  In case all exponents are the same we recover the well-known equation
$$
u_t=\Delta u^m, \quad m>0\,,
$$
which for $m=1$ is just the classical heat equation. For $m\ne 1$ it is a well-studied model for nonlinear diffusion and heat propagation. For $m>1$ the equation is degenerate parabolic and is called the Porous Medium Equation, PME, see \cite{Vlibro}. On the other hand, for $m<1$ the equation is singular parabolic and is called the Fast Diffusion Equation, FDE, see \cite{DK07, VazSmooth}.    The solutions are assumed to be nonnegative; this restriction makes sense in view of applications where $u$ is an evolving mass density.

According to standard terminology, a fundamental solution is a finite mass solution of the Cauchy problem    having the Dirac mass as its initial trace, more precisely $u(x,t)\to  M\,\delta(x)$ as $t\to 0$ in the sense of distributions.  The typical fundamental solution is the one with constant $ M=1$. The concept plays a central role in the theory of linear PDEs.    Fundamental solutions are also important in nonlinear parabolic problems of diffusion type, where they are also called source-type solutions, a main reference being Barenblatt's \cite{Barbk96}, see also \cite{Kamin, Vascppme, Vaz20}. Once constructed, the self-similar fundamental solutions are shown to be the asymptotic attractors of all    nonnegative  solutions with finite mass for a number of relevant PDEs.    It is our purpose to show that this phenomenon occurs in the anisotropic equation  APME.

   Equation \eqref{APM} and similar ones appear for instance in hydrology as simplified models for the motion of water in anisotropic media,   see \cite{Bear, H, S01, S01bis, SJ05, SJ06}. If the conductivities of the media may be different in different directions, the constants $m_i $ in \eqref{APM} may be different from each other. Note that the spatial operator in \eqref{APM} is the sum of independent { 1-dimensional }Laplacians along the different coordinate directions, each  applied to a possibly different power of $u$.
We will consider solutions to the Cauchy problem for \eqref{APM} with nonnegative initial data
\begin{equation}\label{IC}
u(x,0)=u_0(x), \quad x\in \R^{N}.
\end{equation}
We will assume that  $u_0\in L^1(\mathbb{R}^N)$, $u_0\ge 0$, and we put $M:=\int_{\mathbb{R}^N} u_0(x)\,dx$, so-called total mass.    We look for solutions $u\ge 0$.

Experience with the isotropic PME shows that the existence and behaviour of special solutions strongly depends on the range of exponents \cite{VazSmooth},  and so happens with their role in the theory.  In this paper we will focus on the fast diffusion range

\noindent  (H1) \centerline{$0< m_i< 1$ \qquad  for all $i=1,...,N$.}

\noindent Note that this is a condition of ``fast diffusion in all directions'' that is made here for convenience of exposition since it allows for a unified theory with clear-cut results. We will refer to the equation in that range as AFDE. As a natural extension, we also consider at the end of the paper cases where some exponents are 1, i.e., linear diffusion in some directions, but this is not our main interest.  Note that when all     $m_i=m<1$ are equal, we recover the classical (isotropic) Fast Diffusion Equation, and when $m=1$ the classical Heat Equation.

We need a further assumption on the exponents. We recall that in the isotropic fast diffusion equation  (\emph{i.e.,} equation \eqref{APM} with $m_{1}=...=m_{N}=m<1$), there is a well-known \emph{critical exponent},
\begin{equation}\label{m bar}
m_c:=1-\frac{2}{N}, \quad
\end{equation}
such that $m>m_c$ is a necessary and sufficient condition  for the existence of fundamental solutions, see for instance \cite{VazSmooth}.
In the same spirit, in this work we will always assume the average condition

 \noindent  (H2) \centerline{$\displaystyle \overline{m}:=\frac1{N}  {\sum_{i=1}^N m_i}   > m_c\,.$ }

\noindent    This can be written as $\sum_{i=1}^N m_i>N-2$. This condition is crucial in our paper. Indeed, we will show that (H2) alone ensures the existence of the self-similar fundamental solution (SSFS)   in the anisotropic FDE  case. The fact that the fundamental solution has a self-similar form will be a consequence of the analysis we perform, based on the scaling invariance satisfied by the equation. See details in Subsection \ref{ssc.sss}.    Let us point out that condition (H2)  may allow for some $m_i$ to be less than $m_c$ in dimensions 3 or more. In any case we take $m_i>0$. In dimension $N=2$ condition (H2) implies no restriction.

 The problem we discuss came to our attention years ago during a visit of Prof. B.~H. Song to Madrid. He then published a number of works on the issue, mentioned above. Of interest here are   \cite{SJ06} where solutions with finite mass are constructed, and \cite{SJ05} where a fundamental solution is constructed for general initial data, i.e., a solution with a Dirac delta as initial data. It was supposed to be the basis of asymptotic long-time analysis.

 We contribute the missing analysis of self-similarity, which produces a  critical amount of extra information    and paves the way to the asymptotic behaviour. We note the presence of the anisotropy produces several difficulties that cannot be approached by classical tools as in the isotropic case, hence the problem had remained open for all these years. Indeed,  the combination of self-similarity and anisotropy is an uncommon topic in the literature, see an example in \cite{Per2006}, far from our field. However, it is rich in details and consequences.

 Here,    we construct a supporting theory and   prove two main results. Firstly, we establish the existence of a unique fundamental solution of self-similar type (SSFS), one for every mass $M>0$. We do it by using a new fixed-point argument and the mass difference analysis, which are flexible techniques that could be useful in a broad variety of situations. This allows to identify in a very precise way not only the decay and propagation exponents in every direction, but also the whole asymptotic profile $F$ (see details in Section \ref{ssc.sss}), which is shown to be a solution to an anisotropic nonlinear elliptic problem of Fokker-Planck type:
 \begin{equation}\label{StatEq}
\sum_{i=1}^{N}\left[(F^{m_i})_{y_iy_i}+\alpha \sigma_i\left( y_i F\right)_{y_i}
\right]=0.
\end{equation}
 Explicit solutions of this    nonlinear elliptic   equation are not known so far,  but we prove that $F$ is a positive and $C^\infty$ smooth function. The proof of the result relies on tools like a comparison principle and the construction of an    explicit   anisotropic upper barrier,    an important tool   used to have an upper control of general solutions. A specific feature for the fixed point argument is the use of a suitable quantitative positivity lemma for solutions of the rescaled equation which lie below the anisotropic upper barrier at the initial time.
 Furthermore, numerical studies highlighted in Section \ref{se.numer} confirm the nonstandard shape of the self-similar
 profiles $F$ for different choices of the initial data.

The second main result shows the  role of the self-similar solutions we have just constructed as attractors. Thus, we are able to establish the sharp asymptotic convergence  of any nonnegative solution with finite mass towards the self-similar solution with the same mass, this being the other main  result of the paper (see Section 8). In this way we complete for our equation the program outlined by G. Barenblatt in \cite{Barbk96} about scaling,  self-similarity, and intermediate asymptotics.

 The case of partial linear diffusion, where some $m_i=1$, has some special features that we will briefly discuss at the end of the paper. The case where slow diffusion exponents $m_i>1$ appear deserves separate analysis and is not treated in this work.


\subsection{Self-similar solutions}\label{ssc.sss}

We present next in an informal way the main functions to be constructed and studied. The formal justification will be done in this paper in the framework of weak energy solutions. That concept is well known in the theories of nonlinear diffusion, but we will review the needed theory for the reader's benefit in Section \ref{sec.basic}, which contains  existence, uniqueness and useful properties of such solutions for the Cauchy problem. Later results on regularity and positivity, proved in the paper, simplify the issue.

   Let us examine our main construction.   The common type of self-similar solution of equation \eqref{APM} has the form
\begin{equation}\label{sss}
U(x,t)=t^{-\alpha}F(t^{-a_1}x_1,..,t^{-a_N}x_N)
\end{equation}
with constants $\alpha>0$, and $a_1,..,a_n\ge 0$ to be chosen below. We look for this type  as model solutions for our equation  \eqref{APM}.
As announced before, we will restrict the study to nonnegative solutions. \noindent Note that,  after writing $y=(y_1,\cdots,y_N)$ and $y_i=x_i \,t^{-a_i}$, we have
$$
U_t=-t^{-\alpha-1}\left[\alpha F(y)+
\sum_{i=1}^{N}a_iy_i\,F_{y_i}
\right]
$$
and
$$
\sum_{i=1}^{N}(U^{m_i})_{x_ix_i}=\sum_{i=1}^{N}t^{-(\alpha m_i+2a_i)}(F^{m_i})_{y_iy_i}.
$$
Therefore, equation \eqref{APM} becomes
\begin{equation}
\label{Eq1}
-t^{-\alpha-1}\left[\alpha F(y)+
\sum_{i=1}^{N} a_iy_i\,F_{y_i}
\right]=\sum_{i=1}^{N}t^{-(\alpha m_i+2a_i)}(F^{m_i})_{y_iy_i}.
\end{equation}
We see that time is eliminated as a factor in the resulting equation on the condition that:
\begin{equation}\label{ab}
\alpha(m_i-1)+2a_i=1 \quad \mbox{ for all } i=1,2,\dots,  N.
\end{equation}
We also want integrable solutions that will enjoy the mass conservation property, which
implies \ $\alpha=\sum_{i=1}^{N}a_i$.
Imposing both conditions, and putting $a_i=\sigma_i \alpha,$ we determine in a unique way the values for the exponents \ $\alpha$ and $\sigma_i$ (a lucky fact):
\begin{equation}\label{alfa}
\alpha=\frac{N}{N(\overline{m}-1)+2},
\end{equation}
and
\begin{equation}\label{ai} \ 
 \sigma_i= \frac{1}{N}+ \frac{\overline{m}-m_i}{2}.
\end{equation}

\begin{definition} A mass-preserving self-similar solution to \eqref{APM} is just a solution $U$ to \eqref{APM} of the form \eqref{sss}, where $a_i=\alpha \sigma_i$ for all $i=1,\cdots,N$, and $\alpha$ and $\sigma_i$ satisfy \eqref{alfa} and \eqref{ai}.
\end{definition}

 In what follows we will usually skip writing mass-preserving, because all solutions considered in this paper enjoy the property    (unless mention to the contrary).
Observe that by Condition (H2) imposed in the Introduction we have $\alpha>0$,  so that the self-similar solution will decay in time in maximum value like a power of time. This is a typical feature of many  diffusion processes.

As for the $\sigma_i $ exponents,    that control the expansion of the solution in the different coordinate directions with time , we easily see that \ $\sum_{i=1}^{N}\sigma_i=1$, and in particular $\sigma_i=1/N$ in the isotropic case.  Conditions (H1) and (H2) on the $m_i$  ensure that $\sigma_i> 0$. This means that the self-similar solution expands  as time passes, or at least does not contract, along any of the space coordinate variables.

With these choices,  and working again at the formal level for brevity, the \sl profile function \rm $F(y)$   must satisfy {the nonlinear anisotropic stationary equation  \eqref{StatEq} in $\mathbb{R}^N$.}

\begin{proposition}\label{Lem1}
$U(x,t)$ is a self-similar solution to \eqref{APM} of the form \eqref{sss} where $a_i=\alpha \sigma_i$ for all $i=1,\cdots,N$ and $\alpha$ and $\sigma_i$ satisfy \eqref{alfa} and \eqref{ai} if  and only if its profile $F$ satisfies the stationary equation \eqref{StatEq}. Moreover, $\int U(x,t)\, dx =\int F(y)\, dy=M$ \ {\rm for} $t>0.$
\end{proposition}

\noindent {\sl Proof.} Under our choices of exponents $\alpha$ and $\sigma_i$ given by \eqref{alfa} and \eqref{ai}, equation \eqref{Eq1} becomes \eqref{StatEq}. Besides, the conservation of mass follows by a change of variables. \qed

The    profile $F$   is an interesting mathematical object in itself, as a solution of a nonlinear anisotropic Fokker-Planck equation.
It is our purpose to prove that there exists a suitable solution of this elliptic equation, which is the anisotropic version of the equation of the Barenblatt profiles in the standard PME/FDE, cf. \cite{Barbk96, Vascppme, Vlibro}.
  Again, the general theory deals with weak energy solutions, but we will prove later that the self-similar profiles are even $C^\infty$ smooth functions (see Subsection \ref{ssec.reg}).
The solution is indeed explicit in the isotropic case:
 $$
 F(y;m)=\left( C + \frac{\alpha(1-m)}{2mN}|y|^{{2}}\right)^{{{-1}}/(1-m)},
 $$
 with a free constant $C>0$ that fixes the total mass    $M$   of the solution, $C=C(M)$.    An explicitly expression of $C(M)$ is given in \cite{BS} and in \cite{Vlibro}. Moreover, we refer the reader to Subsection 2.3 for more details on the mass changing rule. It is clear that this formula breaks down for $m\le m_c$ (called very fast diffusion range), where many new developments occur, see the monograph \cite{VazSmooth} and papers \cite{BBDGV, BDGV}.    We will not enter into that range and its features. This explains    our insistence   on restriction (H2).

Here is the main result of the present paper,    dealing with   the theory of self-similar solutions.

\begin{theorem}\label{fundamental solution} Under the restrictions (H1) and (H2), for any mass $M>0$ there is a unique self-similar fundamental solution $U_M(x,t)\ge 0$ of equation \eqref{APM} with mass $M$. The profile $F_M$  of such a solution is an SSNI (separately symmetric and nonincreasing) positive function. Moreover,   $0<F_M(y)\le G_k(y)$, for a suitable  choice of the  barrier function $G_k$ given in formula \eqref{G_k}.
\end{theorem}

\noindent{ \bf Remarks.}
1)  We recall that by solution we understand a weak energy solution, see the theory in Section \ref{sec.basic}. In the end, the solution of this class is proved to be smooth, so the weak energy solution is indeed a classical solution of the equation.

\smallskip

2) For the concept of separately symmetric and nonincreasing function (SSNI for short) see Section \ref{ALEKSAND}. The proof of the main theorem will be done in Section \ref{sec.ex.ssfs}.

\smallskip

3) We will not get any explicit formula for $F_{M}$ in the anisotropic case, but we have suitable estimates, in particular regularity and decay in space.  Thus, we get a clean upper bound for the behaviour of $F_{M}$ at infinity:
$F_{M}(y)\le O(|y_i|^{-2/(1-m_{i})})$. Anisotropy will be evident in the graphics of the level lines, see also the Numerical Section \ref{se.numer}.

\smallskip

4)  The existence of a fundamental solution, not necessary self-similar, was proved in \cite{SJ05}
with a different approach. There is to our knowledge no proof of uniqueness for such a general concept of solution.
Uniqueness is a crucial aspect in the study of asymptotic behaviour to be done later.

\smallskip

5) As in the isotropic case, there is an algebraic way to pass from any mass $M_{1}>0$ to another mass $M_{2}>0$, see Subsection \ref{sec scaling}. Thus, all the $F_M$ functions are rescalings of $F=F_1$, of the form $F_M(y)=k F(k^{\nu_i}y_i)$ with suitable constants  $k=k(M)$ and $\nu_i>0$.

The following result shows that self-similar solutions of the type  \eqref{sss} are actually fundamental solutions to \eqref{APM}.
 \begin{lemma}\label{selfsimarefundsolut}
 If    $U(x,t)=t^{-\alpha}F(t^{-a_1}x_1,..,t^{-a_N}x_N)$ is the self-similar function   defined in \eqref{sss}, where $a_i=\alpha \sigma_i$ for all $i=1,\cdots,N$, and also $\alpha$ and $\sigma_i$ satisfy \eqref{alfa} and \eqref{ai}, then it is  a fundamental solution of the Cauchy Problem \eqref{APM}-\eqref{IC} if \ $F\ge0$,  $F\in L^1(\ren)$ and it satisfies equation \eqref{StatEq}.
 \end{lemma}

\noindent {\sl Proof.} We recall that a self-similar fundamental solution with mass $M$ to the Cauchy Problem \eqref{APM}-\eqref{IC} is just a self-similar solution to \eqref{APM} that tends to  the Dirac delta with mass $M$ as time goes to $t=0$ in a suitable weak sense.
Thus, we only have to check the convergence of $U(x,t) $ to $\delta(x)$ in the sense of measures, i.e.
$$
\lim_{t\to 0}\int_{\mathbb{R}^N} U(x,t)\varphi(x)\,dx= M \varphi(0)
$$
for all $\varphi$ continuous, nonnegative and bounded in $\ren$. This follows from the self-similarity formula and the integrability of $F$.
\qed

\medskip

\subsection{Self-similar variables}
In several instances in the sequel it will be convenient to pass equation \eqref{APM}   to self-similar variables,
by zooming the original solution according to the self-similar exponents \eqref{alfa}-\eqref{ai}. More
precisely, the change is done by the set of anisotropic formulas
 \begin{equation}\label{NewVariables}
v(y,\tau)=(t+t_0)^\alpha u(x,t),\quad \tau=\log (t+t_0),\quad y_i=x_i(t+t_0)^{-\sigma_i\alpha} \quad i=1,..,N,
\end{equation}
with $\alpha$ and $\sigma_i$ as calculated before.
We recall that all of these exponents are positive. There is a free time parameter $t_0\ge 0$ (a time shift) that can be used at convenience, normally $t_0=0$ or $t_0=1$. See in this regard the discussion in \cite{VazSmooth}.


\begin{lemma}\label{Lem1}
If $u(x,t)$ is a solution (resp. supersolution, subsolution) of \eqref{APM}, then $v(y,\tau)$ is a solution (resp. supersolution, subsolution) of
\begin{equation}\label{APMs}
v_\tau=\sum_{i=1}^N\left[(v^{m_i})_{y_i y_i}+\alpha \sigma_i \left(\,y_i\,v\right)_{y_i}\right] \quad \text{ in \quad }\mathbb{R}^N\times(\tau_0,+\infty).
\end{equation}
\end{lemma}
This equation will be a key tool in our study. Note that the rescaled equation does not change with the time-shift  $t_0$   but the initial value of the new time does, $\tau_0 = \log(t_0)$. If $t_0 = 0$ then $\tau_0 =-\infty$ and the $v$ equation is defined for all $\tau\in \mathbb{R}$.

We stress that this change of variables preserves the $L^1$ norm: the mass of the $v$ solution at new time $\tau:=\log(t+t_0)\geq\tau_0:=\log t_0$ equals that of the $u$ at the corresponding time $t\geq t_0$:
 $$\int_{\mathbb{R}^N}v(y,\tau) \, dy=\int_{\mathbb{R}^N}u(x,t) \, dx\quad \text{ if }\tau=\log (t+t_0).$$

   We recall that so far the approach is similar to the one used in the works that deal with the isotropic PME and FDE, as explained in great detail in \cite{Vlibro, VazSmooth}.

\subsection{Outline of later sections }\label{ssec.outline}

After the introduction of the problem, conditions, and concept of self-similarity  done in this section, we devote Section \ref{sec.basic} to establish the basic theory  of weak energy  solutions to be used and its main properties. The theory follows ideas used in the isotropic case but there are some special features and derivations that we explain in some  detail and they lead to Theorem   \ref{EUWES}.

Section \ref{sec.upp} contains the construction of the Anisotropic  Upper Barrier, a key tool in the proof of existence of a self-similar fundamental solution. This is followed by two technical sections on Aleksandrov's Principle and local positivity.

After this preparation, we are ready for the statement and proof of existence and uniqueness of a self-similar fundamental solution, contained in the important Section \ref{sec.ex.ssfs}. This proof faces several difficulties that are not found in previous works on degenerate parabolic equations of porous medium or fast diffusion type. A number of novel ideas are introduced; some similar ideas were recently used in \cite{Vaz20}. We also prove monotonicity, positivity and regularity of the profile.

 Section \ref{pos.gen}  deals with the strict positivity of nonnegative solutions, a typical fast diffusion feature.    Regularity follows.


In Section \ref{sec.asymp} we establish the asymptotic behaviour of finite mass solutions, another goal of this paper. Solution is understood in the sense of    Theorem \ref{EUWES} , Section \ref{sec.basic}.

\begin{theorem}\label{thmasympto}
Let $u(x,t)$ be the unique solution  of the Cauchy problem for equation \eqref{APM} with nonnegative initial data $u_{0}\in L^{1}(\R^{N})$  under the restrictions (H1) and (H2). Let $U_{M}$ be the unique self-similar fundamental   solution with the same mass as $u_{0}$. Then,
\begin{equation}\label{L1conv}
\lim_{t\rightarrow\infty}\|u(t)-U_{M}(t)\|_{1}=0.
\end{equation}
The convergence holds in the $L^{p}$ norms, $1\le p <\infty$,  in the proper scale
\begin{equation}\label{Lpconv}
\lim_{t\rightarrow\infty}t^{\frac{(p-1)\alpha}{p}}\|u(t)-U_{M}(t)\|_{p}=0,
\end{equation}
where $\alpha=N/(N(\overline{m}-1)+2)$ is the constant in \eqref{alfa}.
 Finally, under certain conditions of the initial data, convergence    \eqref{Lpconv}   holds also for $p=\infty$,   see Theorem \ref{thm8.2}.
\end{theorem}

 At the time the present article was written, a number of similar ideas were used in the study of nonlocal nonlinear diffusion in \cite{VspLsub} and then they appeared in the simpler study of anisotropic $p$-Laplacian local diffusion in \cite{FVV21}.

As a complement to this information, numerical studies are produced in Section \ref{se.numer} to make clear the effect of anisotropy on the shape of the solutions. We briefly discuss the case of partial linear diffusion in Section \ref{sec.meq1}. The paper ends with a section on comments  and open problems.


\section{Preliminaries. Basic theory}\label{sec.basic}

Even in the case of the isotropic FDE the existence of classical solutions is not granted a priori, so the basic existence and uniqueness theory deals with solutions in some generalized sense using the techniques of Nonlinear Analysis. This approach is followed here. The simplest concept of solution of the APME  is in principle the \sl distributional solution \rm where we consider a function $u\ge 0$ defined in $Q=\mathbb{R}^N\times(0,+\infty)$, such that $u, u^{m_{i}}\in L^{1}_{loc}(Q)$ for all $i=1,...,N$, and equation \eqref{APM}  is solved  the distributional sense, i.e.,
\begin{align}\label{distr.ren}
&\int_{0}^{T}\int_{\ren}u\varphi_{t}\,dx\,dt+\sum_{i=1}^{N}\int_{0}^{T}\int_{\ren}{ u^{m_{i}}\,\varphi_{x_{i}x_{i}}}dx\,dt\\
&=\int_{\ren}u(x,T)\varphi(x,T)dx\nonumber -\int_{\ren}u_{0}(x)\varphi(x,0)dx,
\end{align}
for all the compactly supported test functions $\varphi\in C^{\infty}(\overline Q)$ and every $T>0$.

The need to prove uniqueness and a number of extra properties for the class of solutions we actually use leads in this section to the introduction of more restrictive concept of solution. Thus, for the isotropic PME/FDE the class of \sl mild solutions \rm described by Crandall-Liggett in \cite{CL71} with $L^1$ initial data provides a general concept that enjoys the properties of uniqueness, comparison, smoothing effect, energy estimates and conservation of mass, among others.    See also B\'enilan's approach with his integral solutions   in \cite{Benilan72}. For the application to the theory of the PME we refer to \cite{Vlibro}. 
Like in the PME/FDE the absence of a right-hand side in the equation allows to conclude that a more suitable subclass of distributional solutions enjoys existence and uniqueness as well as good estimates. This is the class of \sl weak energy solutions \rm contained in Theorem \ref{EUWES}.

We recall here that the existence and uniqueness of suitable solutions of our Cauchy problem with integrable nonnegative data was solved by Song and Jian in  \cite{SJ06} after a fundamental solution was constructed in \cite{SJ05}. Thus, their  Theorem 1.2 proves that, under some assumptions on the problem, for any nonnegative $u_0 \in L^1(\mathbb{R}^N)$ {there is a unique function $u$ such that $u, u^{m_{i}}\in L^{1}_{loc}(Q)$ for all $i=1,...,N$, solving equation \eqref{APM} in the distributional sense on $Q=\mathbb{R}^N\times(0,+\infty)$, with the following properties:}

- $u\in C([0,\infty): L^1(\mathbb{R}^N))$, \ $u \in C(Q)\cap L^\infty(\mathbb{R}^N \times [\tau,\infty))$ for each $\tau>0$,

- $u$  takes the initial data  in the sense that $u(x,t)\to u_0(x)$ in $L^1(\ren)$ as $t\to 0$.

- The solution preserves the total mass, $M=\int_{\mathbb{R}^N} u_0(x) \,dx=\int_{\mathbb{R}^N} u(x,t)\,dx$.

They call this type of solution the $ L^1$-regular solution.

\medskip

This is the complete statement of the results    we use in the paper.

 \begin{theorem}\label{EUWES}    Let the exponents $m_i$ satisfy assumptions (H1) and (H2). Then, for any nonnegative $u_0 \in L^1(\mathbb{R}^N)$ {there is a unique function $u\in C([0,\infty): L^1(\mathbb{R}^N))$ such that $u, u^{m_{i}}\in L^{1}_{loc}(Q)$ for all $i=1,...,N$, and equation \eqref{APM} holds in the distributional  sense in $Q=\mathbb{R}^N\times(0,+\infty)$, with the following additional properties:}

1)  $u(x,t) $ is a uniformly bounded function for each $\tau>0$ with an estimate of the form $\|u(\cdot,t)\|_\infty\leq C t^{-\alpha}$. The precise estimate is given in \eqref{Linfty-L1}.

2)  Let $Q_\tau=\ren\times (\tau,\infty)$. We have $\partial_i u^{m_{i}}\in L^{2}(Q_\tau)$ for every $i$ and the energy estimates \eqref{Energywholespace} are satisfied.
 Equation \eqref{APM} holds in the weak sense of  \eqref{weak.ren} applied in $Q_\tau$ for every $\tau>0$.

3) Consequently, the maps $S_t: u_0\mapsto u(\cdot,t)$ generate a semigroup of $L^1$ ordered contractions in $L^1_+(\ren)$. The $L^1$-contraction estimates \eqref{L1_contr} are satisfied. The maximum principle applies.

4) Conservation of mass holds: for all $t>0$ we have \ $\int u(x,t)\,dx=\int u_0(x)\,dx$. Assumption (H2) is crucial.

5) If we start with initial data $u_0\in L^1(\ren)\cap L^\infty(\ren)$ we may also conclude item 2) with $\tau=0$ and $u(x,t)$ is uniformly bounded and continuous   in space and time.

 \end{theorem}

Below, we will follow an approach   to existence that is self-contained. Indeed, we want to establish the existence of  a non-negative solution $u(x,t)$ with nonnegative initial datum $u_0$ by a method of smooth positive approximations for better-behaved approximate problems with bounded positive data. Though this theory is not the main scope of the paper, in view of the previous results by Song et al. and the current state of the theory of nonlinear diffusion equations, we hope it will be enlightening for the reader and  useful in justifying the above items and different results and proofs in what follows.


\subsection{Construction of solutions by approximation}\label{ssec.approx}

We start with initial data $u_0\in L^1(\ren)\cap L^\infty(\ren)$ and construct an ($L^2$) \sl weak energy solution \rm  $u$,  in the sense that $u\in L^{2}(Q)$, $\frac{\partial}{\partial x_{i}}u^{m_{i}}\in L^{2}(Q)$ \, for all $i=1,\cdots,N$  and it satisfies
\begin{align}\label{weak.ren}
&\int_{0}^{T}\int_{\ren}u\varphi_{t}\,dx\,dt-\sum_{i=1}^{N}\int_{0}^{T}\int_{\ren} (u^{m_{i}})_{x_{i}}\varphi_{x_{i}}dx\,dt\\
&=\int_{\ren}u(x,T)\varphi(x,T)dx\nonumber -\int_{\ren}u_{0}(x)\varphi(x,0)dx,
\end{align}
for all the test functions $\varphi\in C^{1}(Q)$ with $\varphi(x,t)\rightarrow 0$ as $|x|\rightarrow \infty$ for all $t$. Moreover, these solutions will enjoy the \sl energy estimates   \rm
\begin{equation}\label{Energywholespace bis}
\begin{split}
 4 \sum_{j=1}^N\frac{m_im_j}{(m_i+m_j)^2}\int_0^T\int_{\ren}& \left|\frac{\partial}{\partial x_j}\left(u^{\frac{m_i+m_j}{2}}\right)\right|^2 \, dx\, dt
\\
\leq&\int_{\ren}\left[
\frac{1}{m_i+1}{u_0}^{m_i+1}
\right]\, dx
-
\int_{\ren}\left[
\frac{1}{m_i+1}{u}^{m_i+1}(x,T)
\right]\, dx
\end{split}
\end{equation}
for all $i=1,...,N$ and $T>0$.  Since all the terms in the left-hand side are nonnegative we get  in particular
\begin{equation}\label{Energywholespace}
\int_0^T\int_{\ren}
\left|\frac{\partial}{\partial x_i}u^{m_i}\right|^2 \, dx\, dt
\leq\int_{\ren}\left[
\frac{1}{m_i+1}{u_0}^{m_i+1}
\right]\, dx - \int_{\ren}\left[
\frac{1}{m_i+1}{u}^{m_i+1}(x,T)
\right]\, dx
\end{equation}
for all $i=1,...,N$ and $T>0$. This estimate is called the energy estimate.

(i) \textit{Sequence of approximate Cauchy-Dirichlet problems in a ball. } We now consider the following sequence of approximate Cauchy-Dirichlet problems
\begin{equation}\label{PCD}\tag{$P_n$}
\left\{
\begin{aligned}
    &(u_n)_t=\sum_{i=1}^N(u_n^{m_i})_{x_ix_i} & \quad \text{in }Q_n:=B_n(0)\times (0,+\infty), \\
    &u_n(x,0)={u_0}_n(x) &\quad  \text{for }|x|\leq n, \\
   &u_n(x,t)=0 &\quad  \text{for }|x|=n, t\geq0,
\end{aligned}
\right.
\end{equation}
where $B_n(0):=\{x:|x|<n\}$, and \ ${u_0}_n\ge 0$ is a suitable approximation of $u_0$ in $B_n(0)$. This is a rather standard
method. However, solving the problem in this formulation encounters the difficulty that the equation is not uniformly parabolic at the level $u=0$ because of the lack of regularity of the powers $u^{m_i}$  at such level, i.e.,  the diffusion coefficients $m_iu^{m_i-1}$ blow up  when $u \to 0$. This is well-known in the isotropic PME case, see Theorem 5.5 in \cite{Vlibro}. To avoid this degenerate parabolic character we will use a rather standard regularization approach. Instead of solving $(P_n)$ we begin by constructing a sequence of approximate initial data $u_{0,n,\varepsilon}$ which do not take the value $u = 0$. We will avoid the singularity of the equation by moving up the initial and boundary data. Thus, we replace problem \eqref{PCD} by
\begin{equation}
\label{PCD appr}
\tag{$P_{n,\varepsilon}$}
\left\{
\begin{aligned}
    &(u_{n,\varepsilon})_t=
    \sum_{i=1}^N\left(a_\varepsilon^i(u_{n,\varepsilon})(u_{n,\varepsilon})_{x_i}\right)_{x_i} & \quad \text{in }Q_n, \\
    &u_{n,\varepsilon}(x,0)=u_{0,n,\varepsilon}(x) &\quad  \text{for }|x|\leq n, \\
   &u_{n,\varepsilon}(x,t)=\varepsilon &\quad  \text{for }|x|=n, \ t\geq0,
\end{aligned}
\right.
\end{equation}
where, for definiteness, we put \ $u_{0,n,\varepsilon}=u_{0n}+\varepsilon$. We recall that we are assuming $u_0$ bounded.
The new diffusion coefficients $a_\varepsilon^i(u_{n,\varepsilon})$ are chosen to be bounded and uniformly bounded from below and the nonlinearities are such that $a_\varepsilon^i(z)=m_{i}z^{m_{i}-1}$ for $z\in[\varepsilon, \sup u_{0}+\varepsilon]$.

Since problem \eqref{PCD appr} is uniformly parabolic, we can apply  the standard quasilinear theory,   see \cite{LSU} , to find a unique solution $u_{n,\varepsilon}(x,t)$, which is bounded from below by $\varepsilon>0$ in view of the Maximum Principle. Moreover, the solutions $u_{n,\varepsilon}$ in this step are $C^\infty(Q_n)$ by bootstrap arguments based on repeated differentiation and interior regularity results for parabolic equations. Using again the Maximum Principle we conclude that
 $\varepsilon\le u_{n,\varepsilon} \le \sup u_{0}+\varepsilon$.
It follows by the definition of $a_\varepsilon^i$ that we can then replace
$a_\varepsilon^i(u_{n,\varepsilon})\,\partial_{x_i}(u_{n,\varepsilon})$  by $\partial_{x_i}(u^{m_{i}})$ in the equation of \eqref{PCD appr}, as planned from the beginning.

  Moreover, we get monotonicity in time for different norms of these solutions. Indeed, we easily obtain for $u=u_{n,\ve}$
\begin{equation}\label{monot.norm}
\frac{d}{dt}\int_{B_n(0)}(u-\ve)^p(x,t)\,dx=-p(p-1)\sum_{i=1}^{N} m_i \int_{B_n(0)} (u-\ve)^{p-2}u^{m_i-1}\left|u_{x_{i}}\right|^{2}\,dx\le 0,
\end{equation}
from which we conclude that the $L^p$ norms $\|u(\cdot,t)-\ve\|_p$ are nonincreasing in time for every $p> 1$, and in the limit also for $p=1$ and $p=\infty$. This will be recalled and extended in Proposition \ref{decay of the $L^p$} and other places.

In order to get energy estimates that are uniform in $\ve$ and $n$, we proceed as follows. We multiply the equation in \eqref{PCD appr} by $\eta_\varepsilon=u_{n,\varepsilon}^q-\varepsilon^{q}$ with $q=m_i$ for some $i$. Integrating by parts in $B_n(0)\times(0,T)$ and recalling the non-negativity of the solutions, we get
\begin{equation}\label{E1}
\begin{split}
4 \sum_{j=1}^N\frac{m_im_j}{(m_i+m_j)^2}\int_0^T\int_{B_n(0)}&\left|\frac{\partial}{\partial x_j}\left(u_{n,\varepsilon}^{\frac{m_i+m_j}{2}}\right)\right|^2 \, dx\, dt
\\
\leq&\int_{B_n(0)}\left[
\frac{1}{m_i+1}{u_{0,n,\varepsilon}}^{m_i+1}
-\varepsilon^{m_i}u_{0,n,\varepsilon}
\right]\, dx
\\
&-
\int_{B_n(0)}\left[
\frac{1}{m_i+1}{u}_{n,\varepsilon}^{m_i+1}(x,T)
-\varepsilon^{m_i}{u}_{n,\varepsilon}(x,T)
\right]\, dx
\end{split}
\end{equation}
for each $i$  with $T>0$. We can sum in $i=1,...,N$ to get a joint inequality.

\medskip

(ii)\textit{ Passage to the limit as $n\rightarrow\infty$ }.   We let the ball $B_n(0)$ expand into the whole space for fixed $\ve>0$. The family $\{u_{n,\varepsilon}: n\ge 1\}$ is uniformly bounded in $Q_n$ and also uniformly away from 0.
We recall that each $u_{n,\varepsilon}$  is a non-negative solution of problem \eqref{PCD appr}. Since $u_{n,\varepsilon}(x,t)\leq u_{n+1,\varepsilon}(x,t)$ on the boundary of the cylinder $Q_n$,  applying the classical comparison principle we get \ $u_{n,\varepsilon}(x,t)\leq u_{n+1,\varepsilon}(x,t)\text{ \ in } Q_{n}$. Thus, we obtain the monotonicity of $u_{n,\varepsilon}$ in $n$ and we are able to pass to the limit as $n\rightarrow\infty$.  Moreover, estimate \eqref{E1} guarantees a uniform estimate of $\partial_{x_i}(u_{n,\ve}^{m_i})$ in $L^2(B_n(0)\times(0,\infty))$ for all $i$. Using the interior regularity for uniformly parabolic equations (already quoted), we may pass to the limit and get
$$
u_\ve (x,t)=\lim_{n\to\infty} u_{n,\varepsilon}(x,t),
$$
which is a classical solution of \eqref{APM} in the whole space. Since
$$
\varepsilon\leq u_{n,\varepsilon}(x,t)\leq \sup u_0 +\varepsilon\quad \text{ in } Q_n
$$
the same inequalities apply to $u_\ve(x,t)$ in $Q=\R^N\times (0,\infty)$.

It is easy to see that the monotonicity in time of the norms  $\|u-\ve\|_p$ is kept in the limit. In order to pass to the limit in the energy inequalities we have to find an expression in the right-hand side of \eqref{E1} that is uniformly bounded in $n$ and allows for passage to the limit $n\to \infty$. We proceed as follows with $u=u_{n,\ve}$:
the mentioned right-hand term in \eqref{E1} can be written as
\begin{equation}\label{E2}
\begin{split}
\dots \leq& \frac{1}{m_i+1} \int_{B_n(0)} ( u_{0,n,\varepsilon}^{m_i}-\ve^{m_i})\, u_{0,n,\varepsilon} \, dx
\\
& + \frac{m_i}{m_i+1}
\int_{B_n(0)}\ve^{m_i} ({u}_{n,\varepsilon}(x,T)-{u}_{n,\varepsilon}(x,0))\, dx-\frac{1}{m_i+1} \int_{B_n(0)} ({u}^{m_{i}}_{n,\varepsilon}(x,T)-\varepsilon^{m_{i}}){u}_{n,\varepsilon}(x,T)\, dx\\
&=:I_{1}+I_{2}+I_{3}.
\end{split}
\end{equation}
  We may neglect the last term since the integrand is non-negative.
 We can pass to the limit as $n\to\infty$ in the remaining expression, because both integrals $I_1, I_2$ can be bounded by a constant not depending on $n$. Indeed, by the mean value theorem (recall that $u_{0,n,\varepsilon}=u_{0n}+\varepsilon$) we have
\[
I_{1}\leq \frac{1}{\varepsilon}\frac{m_{i}}{m_{i}+1}\int_{B_{n}(0)}(u_{0n}+\varepsilon)^{m_{i}}u_{0n}\,dx\leq C_1(\varepsilon,N)
\]
Moreover, the $L^{p}$ monotonicity of the norm \eqref{monot.norm} gives $I_{2}\leq0$.

Now, since $T$ is arbitrary, it follows that $\{\frac{\partial}{\partial x_i}u_{n,\varepsilon} ^{m_i}\}$ is uniformly bounded in $L^2(Q)$ for all $i$. Therefore, a subsequence of it converges to some limit $\psi_{\varepsilon,i}$ weakly in
$L^2(Q)$. Since\\ $u_{n,\varepsilon}\rightarrow u_{\varepsilon}$ for $n\rightarrow\infty$ everywhere, we can identify $\psi_{\varepsilon,i}=\frac{\partial}{\partial x_i}u_{\varepsilon}^{m_i}$ in the
sense of distributions. Therefore we can pass to the limit in the energy estimate \eqref{E1} and obtain

\begin{equation}\label{Eeps}
\begin{split}
4 \sum_{j=1}^N\frac{m_im_j}{(m_i+m_j)^2}\int_0^T\int_{\R^{N}}&\left|\frac{\partial}{\partial x_j}\left(u_{\varepsilon}^{\frac{m_i+m_j}{2}}\right)\right|^2 \, dx\, dt
\\
\leq&\int_{\R^{N}}\left[
\frac{1}{m_i+1}({u_0}+\varepsilon)^{m_i+1}
-\varepsilon^{m_i}({u_0}+\varepsilon)
\right]\, dx
\\
&-
\int_{\R^{N}}\left[
\frac{1}{m_i+1}{u}_{\varepsilon}^{m_i+1}(x,T)
-\varepsilon^{m_i}{u}_{\varepsilon}(x,T)
\right]\, dx.
\end{split}
\end{equation}

\medskip

(iv) \textit{ Passage to the limit as $\ve \to 0$ }. We notice that the family $\{u_{\varepsilon}\}$ is monotone in $\ve$  by the construction of the initial data. We may then define the limit function
\begin{equation}\label{un}
u(x,t)=\lim_{\varepsilon\rightarrow0} u_\varepsilon(x,t)
\end{equation}
as a monotone limit in $ Q$ of bounded non-negative (smooth) functions. We see that $u_\varepsilon$ converges
to $u$ in a.e. in $Q$ and also in $L^\infty (0,T:L^p(K))$ for every $1 \leq p< \infty$, every finite $T$ and compact set $K\subset \R^N$. We want to show that the limit $u$ is a weak
energy solution of Problem  \eqref{PCD} with initial datum $u_0$. Passing to the limit in \eqref{Eeps} as $\varepsilon\rightarrow0$, we get the following anisotropic energy inequalities:
\begin{equation}\label{E2b}
\begin{split}
 4 \sum_{j=1}^N\frac{m_im_j}{(m_i+m_j)^2}\int_0^T\int_{\R^N}&\left|\frac{\partial}{\partial x_j}\left(u^{\frac{m_i+m_j}{2}}\right)\right|^2 \, dx\, dt+
\frac{1}{m_i+1} \int_{\R^N} { u }^{m_i+1}(x,T) \, dx
\\
&\leq
\frac{1}{m_i+1} \int_{\R^N} {u_0}^{m_i+1} \, dx
\end{split}
\end{equation}
for all $i$   and $T>0$.
Finally, since $u_\varepsilon$ is a classical solution, it clearly is a weak solution with
initial datum ${u_0}_\varepsilon$. Letting $\varepsilon \rightarrow 0$ in the weak formulation
we get that $u$ is a weak
solution  \eqref{PCD} with initial datum $u_0$, in the sense that $u$ satisfies the equality
\begin{equation}\label{weak.sol}
\begin{split}
&\int_{0}^{T}\int_{B_{n}(0)}u\varphi_{t}\,dx\,dt-\sum_{i=1}^{N}\int_{0}^{T}\int_{B_{n}(0)} (u^{m_{i}})_{x_{i}}\varphi_{x_{i}}dx\,dt\nonumber =
\\& \int_{B_{n}(0)}u(x,T)\varphi(x,T)dx-\int_{B_{n}(0)}u_{0}(x)\varphi(x,0)dx,
\end{split}
\end{equation}
for all the compactly supported  test functions $\varphi\in C^{1}(\R^N\times [0,\infty])$.

We remark that all the solutions $u_{\varepsilon}$ are bonded in $L^p(Q_T)$, $Q_T=\Omega\times (0,T)\blue$ for all $1\le p\le \infty$, while $\nabla u\in L^2(Q_T)$ since for all $i$
$$
 \partial_{x_{i}}u_{\varepsilon}=\frac{1}{m_{i}}\,u_{\varepsilon}^{1-m_i} \, (u_{\varepsilon}^{m_i})_{x_i}\in L^2(0,T:L^2(\Omega)).
$$
As for regularity in the time derivative we get weaker information. We have  $u_t=\sum_i \partial_i w_i$ with $w_i=(u^{m_i})_{x_i}, $  hence we have $u_t\in L^2(0,T:X)$ with $X=W^{-1,2}(\Omega)$, for any open bounded set $\Omega$ of $\R^{N}$.
Then the famous Aubin-Lions-Simon lemma \cite{Aub1963, Sim86}  that, in an adapted form, says that if a  sequence  $u_\varepsilon$ is bounded in $L^2(0,T:H^{1}(\Omega))$ and $\partial_t u_n$ is bounded in $L^p(0,T:X)$, $p\geq 1$ with any $X$ some Banach space containing $L^2(\Omega)$, then it is precompact in $L^2([0,T]:L^2(\Omega))$. If $p=\infty$ it is precompact in $C([0,T]:L^2(\Omega))$. We conclude that $u=\lim\limits_{\varepsilon\to 0} u_{\varepsilon}\in L^2([0,T]:L^2(\Omega))$ with a.e. limit.



\medskip

  \noindent {\bf Conclusion.} We may call this limit the \sl constructed  solution\rm, which is a weak energy solution obtained as a monotone limit of positive classical solutions. In the literature we find the name SOLA (i.e., solution obtained as limit of approximations) for a similar situation, see for instance \cite{DallA, KM}. Note that this object has been defined in a unique way by the above construction. Besides, such a type of limit solution is sometimes called in the technical literature a maximal solution, a property derived from our way of approximation. We are going to prove next a more general uniqueness result so that these labels will not be needed. Note that up to this point we are dealing with bounded solutions.

%

\subsubsection{ Comparison of constructed solution with the $L^1$ regular solution} \label{subsec_L1 reg}

Let us go back to the results of \cite{SJ06}. The existence of an $L^1$ regular solution is solved there
by a approximation method from below that starts by treating the case of initial data that are bounded and also compactly supported. For such data the two still missing properties: $u\in C([0,\infty): L^1(\mathbb{R}^N))$ and conservation of mass are proved. The proofs depend on the existence of an integrable barrier, see \cite[Theorem 1.2]{SJ06}. The uniqueness of the $L^1$ regular solution follows.

Now we want to compare the $L^1$ regular solution for an initial datum $u_0\ge 0$ with is bounded and compactly supported in $B_R(0)$, let us call it $\underline{u}$, with the $\ve$ approximation $u_\ve$ of our constructed solution, called here $u_c$. The idea is to use the fact that $u_\ve$ is classical and a strict supersolution for our data so it must be easy to put it on top of other solutions. Here is an argument: we recall (\cite{S01, S01bis, SJ06}) that the $L^1$ regular solution $\underline{u}$ can be obtained as  the limit of a sequence $u_k$  of classical approximations to the equation with initial data $u_k(x,0)=\min\{u(x), 1/k\}$ in some ball $B_{R_k}(0)$, $R_k\ge R$, and boundary data $u_k=1/k$ on  $|x|=R_k$. We choose $R_k\to\infty$. By a standard comparison theorem, using the fact that $u_\ve(x,t)\ge \ve$ everywhere, we conclude that for all large $k$ we have $u_k\le u_\ve$ everywhere in
$Q_k=B_{R_k}\times (0,\infty)$. In the limit $k\to \infty$ we get \ $\underline{u}\le u_\ve$ for every $\ve>0$, and the inequality holds in $Q$. Letting now $\ve\to 0$,
$$
\underline{u}(x,t)\le u_c(x,t).
$$

But conservation of mass for the first solution implies that mass must also be conserved for the second since mass cannot increase. Therefore, $\int (u_c(x,t) - \underline{u}(x,t))\,dx=0$ for all $t$. But then $\underline{u}\equiv u_c$.

This equivalence is extended by approximation and passage to the limit to all non-negative bounded data $u_0$. We conclude that the constructed weak energy solutions are the same as the $L^1$ regular solutions. We have uniqueness of such solutions. In what follows we may simply refer to them as \sl the \rm solutions.

Finally, let us recall that  by \cite[Theorem 1]{H} we have plain continuity, $u\in C(Q)$. Let us also add that at all points where $u_0$ is continuous, the constructed solution also is by a simple barrier argument based on comparison that we leave to the reader.

\newpage

\subsection{ $T$-contraction in $L^1$ norm} \label{subsec_contraction}

Let us first recall  the property the $L^p$ boundedness, that follows from formula \eqref{monot.norm} in the limit, plus conservation of mass for $p=1$.

  \begin{proposition}\label{decay of the $L^p$}
Let $u$ be the constructed solution with $u_0\in L^{1}(\mathbb{R}^N)\cap L^p(\mathbb{R}^N)$ for $p\in[1,+\infty]$. Then, $u(t)\in L^{p}(\mathbb{\R^{N}})$ and
\begin{equation}
\|u(t)\|_{p}\leq \|u_{0}\|_{p}.\label{boundLpnormindata}
\end{equation}
Under conditions (H1) and (H2) we have equality for $p=1$.
\end{proposition}

The next result shows that the set of constructed solutions enjoys the property of $L^1$ contraction in time in the strong form proposed by B\'enilan \cite{Benilan72} as $T$-contraction, a property that implies comparison.

\begin{theorem}\label{contraction}
For every two constructed  solutions $u_1$ and $u_2$ to \eqref{APM} with respective initial data $u_{0,1}$ and $u_{0,2}$ in $L^{1}(\R^{N})\cap L^\infty(\R^{N})$ we have
\begin{equation}\label{L1_contr}
\int_{\R^{N}} (u_1(t)-u_2(t))_+\,dx\le \int_{\R^{N}} (u_{0,1}-u_{0,2})_+\,dx\,.
\end{equation}
In particular, if $u_{0,1}\le u_{0,2}$ for a.e.  $x$, then for every $t>0$ we have $u_1(t)\le u_2(t)$ for all $x\in \R^N$ and $t>0$.
\end{theorem}

\begin{proof} (i) The proof follows a rather typical idea for estimating evolution in $L^1(\R^{N})$. See the arguments from \cite[Prop. 9.1]{Vlibro} in the isotropic case. Formally, we may proceed as follows. This will work for classical solutions.

\begin{lemma}\label{comparisboundedunbound} Suppose that $u_1$ and $u_2$ are classical  nonnegative solutions to \eqref{APM} defined in a bounded {or unbounded} spatial domain $\Omega$, with smooth boundary, living for a time interval $0\le t\le T$  and suppose that $u_{1}\le u_{2}$ on $\partial   \Omega$. Then the contraction result holds. If  we have \, $\int_{\Omega} (u_{0,1}-u_{0,2})_+\,dx=\infty$ there is no assertion.
\end{lemma}

\noindent {\sl Proof of the lemma.}
 Let $p=p(s)$  be a smooth approximation of the positive part of the sign function $\text{sign}(s)$, with $p(s)=0$ for $s\leq0$, $0\leq p(s)\leq 1$ for all $s\in\R$ and $p^{\prime}(s)\geq0$ for all $s\geq0$. Let us multiply \eqref{APM} by
\ $\varphi(x,t)=p(u_1-u_2)$ \ and integrate over $\Omega $, for each solution $u_{1}$, $u_{2}$. After subtracting the resulting equations,  we then have
\[
\int_{\Omega }(u_{1}-u_{2})_{t}\, p(u_{1}-u_{2})\,dx=\sum_{i=1}^N\int_{\Omega }(u_{1}^{m_i}-u_{2}^{m_{i}})_{x_i x_i\,}p(u_1-u_2)\,)dx.
\]
Letting now $p$ tend to $\text{sign}^{+}$ and observing that
\[
\frac{\partial}{\partial t}(u_{1}-u_{2})_{+}=\text{sign}^{+}(u_{1} - u_{2}) \frac{\partial }{\partial t}(u_{1}-u_{2}),
\]
we get after performing the time integration
\[
\frac{d}{dt}\int_{\Omega }(u_{1}(t)-u_{2}(t))_{+} \,dx=\sum_{i=1}^N\int_{\Omega }(u_{1}^{m_i}-u_{2}^{m_{i}})_{x_i x_i}\,\text{sign}^{+}(u_1-u_2)\,dx.
\]
Now, the well-known Kato's inequality implies that for all $i=1,\ldots,N$
\[
\partial_{x_{i}x_{i}}(u_{1}^{m_i}-u_{2}^{m_{i}})_{+}\geq
[\text{sign}^{+}(u_{1}^{m_i}-u_{2}^{m_{i}})](u_{1}^{m_i}-u_{2}^{m_{i}})_{x_ix_i}\,,
\]
  and we recall that \ $\text{sign}^{+}(u_{1}-u_{2})=\text{sign}^{+}(u_{1}^{m_i}-u_{2}^{m_{i}})$. Thus,
\begin{align}\label{8}
\frac{d}{dt}\int_{\Omega }(u_{1}(t)-u_{2}(t))_{+}\,dx& \leq\sum_{i=1}^N\int_{\Omega }\partial_{x_{i}x_{i}} (u_{1}^{m_i}-u_{2}^{m_{i}})_{+}\,dx
\nonumber
\\
&=\sum_{i=1}^N\int_{\partial\Omega }\partial_{x_{i}}(u_{1}^{m_i}-u_{2}^{m_{i}})_{+}\,\nu_i\,dx,\nonumber
\end{align}
where $\nu_i$ is the $i$.th component of the  external normal to $\partial\Omega$.
The last integrand is negative for all $i$ since $(u_{1}^{m_i}-u_{2}^{m_{i}})_{+}\ge0$ in $\Omega$ and it vanishes on the boundary.  We observe that
\[
\partial_{x_{i}}(u_{1}^{m_i}-u_{2}^{m_{i}})_{+}\,\nu_i= |\nu_{i}|^2\partial_{\nu}(u_{1}^{m_i}-u_{2}^{m_{i}})_{+}\le 0.
\]
 \ Therefore,
$$
\frac{d}{dt}\int_{\Omega }(u_{1}(t)-u_{2}(t))_{+}\,dx \le 0\,.
 $$
This ends the proof for classical solutions. In fact, we get for all $0\le s\le t$
\begin{equation}\label{L1_contr.st}
\int_{\Omega} (u_1(t)-u_2(t))_+\,dx\le \int_{\Omega} (u_1(s)-u_2(s))_+\,dx\,.
\end{equation}
\hskip 15cm\qed

\medskip

(ii) Using this lemma the result will be justified for the constructed weak solutions. Recalling the approximation procedure of Section \ref{ssec.approx} we will first prove the result for the solutions of the approximated problems ($P_{n,\varepsilon}$) posed in the ball $\Omega=B_n(0) $ with data $\overline u_{0,1}(x)=u_{0,1}(x)+\ve$ and $\overline u_{0,2}(x)=u_{0,2}(x)+\ve$. Recall that they are smooth solutions of equation \eqref{APM} with the same boundary data. Hence, by the previous result we have
for all $0\le s\le t$
\begin{equation*}
\int_{B_n(0) } (\overline{ u}_1(t)-\overline{ u}_2(t))_+\,dx\le \int_{B_n(0) } (\overline{ u}_1(s)-\overline{ u}_2(s))_+\,dx\,.
\end{equation*}
It is easy to justify that we can pass to the limit first in $n \to\infty$ and then in $\ve\to 0$ to get the same result
 \eqref{L1_contr} for the constructed solutions in $\R^N$.
\end{proof}

\begin{remark} {\rm To obtain this result we do not need the solutions to be the limit of classical solutions. In order to justify the first formal step (i) the solutions only have to be weak solutions with energy estimates and all involved derivatives in $L^1_{loc}(\Omega\times(0,T))$ with the equation satisfied a.e. in $Q=\Omega\times(0,T)$ so that Kato's result can be used. This is usually called a \sl strong solution\/.}\end{remark}

\begin{remark}
In the next section we shall need Lemma \ref{comparisboundedunbound} applied also to equation \eqref{APMs} in the self-similar variables, which is true. Moreover, the proof of  Lemma \ref{comparisboundedunbound} holds to compare a supersolution and a subsolution as well.
\end{remark}

\subsection{Theory in $L^1$. The $L^1$ to $L^\infty$ estimate }

The main step into a theory with unbounded data $u_0\in L^1(\R^N)$ is a result which is usually known as the $L^1$ to $L^\infty$ smoothing effect.

\begin{theorem}\label{L1LI} If $u_0\in L^1(\mathbb{R}^N)\cap L^\infty(\mathbb{R}^N)$, then the constructed solution $u$ to \eqref{APM}-\eqref{IC} under assumptions (H1) and (H2)   satisfies
 \begin{equation}\label{Linfty-L1}
 \|u(t)\|_\infty\leq C t^{-\alpha}\|u_0\|_1^{2\alpha/N}\quad \forall t>0,
 \end{equation}
where the exponent $\alpha$ is defined in \eqref{alfa} and $C=C(N,m_1,...,m_N)$.
 \end{theorem}

   This important result can be found in \cite{SJ06}. For our  proof see the Appendix.

\medskip

\noindent {\bf Definition of solution with $L^1$ data.} Once we know  bound \eqref{Linfty-L1} with a right-hand side that does not involve the $L^\infty$ norm, we may construct the solution of equation \eqref{APM} in $Q=\R^{N}\times (0,\infty)$ with initial data $u_0\in L^1(\ren)$ as the monotone limit of the approximate solutions $u_k$ with initial data $u_{0,k}$, $k\ge 1$ such that
$u_{0,k}=\min\{u_0(x),k\}.$  Then,
 \begin{equation}\label{L1.def}
u(x,t)=\lim_{k\to\infty} u_k(x,t).
 \end{equation}

\begin{theorem}\label{L1LIbis} The solution defined in \eqref{L1.def} satisfies the conditions of Theorem \ref{EUWES}.
 \end{theorem}

 \begin{proof} The limit is obtained as a non-decreasing limit of bounded functions in $C([0,\infty]:L^1(\R^N))$, hence it belongs to the same class. We leave the rest of the details to the reader or see  \cite{SJ06}.  \end{proof}

\medskip

\begin{remark}From Proposition \ref{decay of the $L^p$} and Theorem \ref{L1LI} we have that for $u_0\in L^1\cap L^\infty$, the rescaled evolution solution $v$ of \eqref{NewVariables} is uniformly bounded in time: indeed, for a fixed $\tau_{1}>0$, Theorem \ref{L1LI} implies, for $\tau>\tau_{1}$,
\[
|v(y,\tau)|\leq C(\tau_1)\|u_{0}\|_{1}^{2\alpha/N}
\]
while Proposition \ref{decay of the $L^p$} yields, for $\tau\leq \tau_{1}$,
\[
|v(y,\tau)|\leq C_{1}(\tau_1)\|u_{0}\|_{\infty}.
\]
We will take into account the  dependence of the constants on $\tau_1$ in what follows.
\end{remark}


\subsection{Scaling}\label{sec scaling} Equation \eqref{APM} is invariant under the scaling transformation
\begin{equation}\label{scal.tr}
\widehat u(x,t)=k^{\alpha}u( k^{a_1}x_1,\cdots,  k^{a_N}x_N  ,kt), \quad k>0,
\end{equation}
assuming that \eqref{ab} holds.
 This is of course related to self-similarity. But we can have other choices {different from \eqref{alfa} and \eqref{ai}}. Suppose we put $a_i=\alpha \omega_i$ and
$$
\omega_i(c)= \frac{c}{N}+\frac{\overline m-m_i}2
$$
for some $c>0$. Then $\sum_i \omega_i(c)= c$ and we can get
$$
\alpha(c) =\frac{1}{\overline m -1+ (2c/N)}
$$
For $c= 1$ we retrieve the old scaling exponents that conserve mass {(see \eqref{alfa} and \eqref{ai})}. Indeed,  conservation of mass does not hold unless $c=1$ since
$$
{M(\widehat u):=}\int_{{\mathbb{R}^N}} \widehat u(t)\,dx= k^{\alpha(c)[1-\sum_i \omega_i(c)]}\int_{{\mathbb{R}^N}} u(kt)\,dx,
$$
hence, $M(\widehat u)=k^{\alpha(c)(1-c)}M(u)$.

\medskip

\noindent $\bullet$ {\bf Scaling for the stationary equation}.
The following transformation changes (super) solutions into new (super) solutions of the stationary equation \eqref{StatEq} and it also changes the mass. We put
\begin{equation}\label{Tk}
{\mathcal T}_k F(y)=F_k(y)=k F(k^{\nu_i}y_i)\,.
\end{equation}
The equation is invariant under this transformation  if  \ $m_i +2\nu_i= 1$ for all $i$, hence $\nu_i=(1-m_i)/2$. Note that this changes the mass ({or the $L^1$ norm})
\begin{equation}\label{Change mass}
\int_{{\mathbb{R}^N}} F_k(y)dy= k\int_{{\mathbb{R}^N}} F(y_i\,k^{\nu_i})\,dy=k^{\beta} \,\int_{{{\mathbb{R}^N}}} F(z)\,dz
\end{equation}
where
\begin{equation}\label{beta}
\beta=1-\sum_i\nu_i=1-N(1-\overline {m})/2=\frac{N}{2}(\overline {m}-m_c)\in (0,1).
\end{equation}
Changing $ F_1=\overline{F}$ into the rescaled version
${\mathcal T}_k F_1$ we can make ${\mathcal T}_k r$ (where $r$ is the radius of the anisotropic domain) as small as we want, and both the mass and the maximum of  ${\mathcal T}_k F_1$ will grow  according to the rates $k^\beta$ and $k$ respectively.  This transformation will be used in the sequel to reduce the calculations with self-similar solutions to the case of mass $M=1$.

\medskip

\section{Anisotropic upper barrier construction}\label{sec.upp}
We first observe that our hypotheses (H1), (H2) and formula \eqref{ai} guarantee that
\begin{equation}\label{33}
\frac{1}{\sigma_i}<\frac{2}{1-m_i}.
\end{equation}

The construction of an upper barrier in an outer domain will play a key role in the proof of existence of the  self-similar   fundamental solution in Section \ref{sec.ex.ssfs}.

\begin{proposition}\label{P11}
 Let us assume $m_i<1$ for all i.    The function
\begin{equation}\label{outer.barr bis}
\overline{F}(y)= \left( \sum_{i=1}^N \gamma_i|y_i|^{\frac{2}{1-m_i}} \right)^{-1}
\end{equation}
with
\begin{equation}\label{gammai}
 0<   \gamma_i
\leq\left[\frac{\alpha}{N}\left(\min_i\{\sigma_i\frac{2}{1-m_i}\}-1\right)\frac{(1-m_i)^2}{4m_i(m_i+1)}
\right]^{\frac{1}{1-m_i}}
\end{equation}
is a weak supersolution to \eqref{StatEq} in ${\mathbb{R}^N\setminus B_R(0)}$ and a classical supersolution  in $\mathbb{R}^N\setminus\left\{0\right\}$, with $B_R(0)$ being a  any   ball of radius $R>0$. Moreover,   $\overline{F}\in L^p(\mathbb{R}^N\setminus B_R(0))$ for every $p\geq1$.
\end{proposition}

We say that $\overline{F}$ is a weak (energy) supersolution to \eqref{StatEq} in ${\mathbb{R}^N\setminus B_R(0)}$ if $\overline{F}\in L^2 (\mathbb{R}^N\setminus B_R(0))$,
$(\overline{F}^{m_{i}})_{y_{i}}\in L^2(\mathbb{R}^N\setminus B_R(0))$ for all $i=1,...,N$ and the following inequality holds
\[
\sum_{i=1}^{N}\int_{\mathbb{R}^N\setminus B_R(0)}\left[(\overline{F}^{m_i})_{y_i}\varphi_{y_{i}}+\alpha \sigma_{i}  y_{i} \overline{F}\, \varphi_{y_i}\right]dy\geq 0
\]
for all the nonnegative functions $\varphi\in C_{c}(\mathbb{R}^N\setminus B_R(0))$.  A classical supersolution occurs when
$$
\sum_{i=1}^{N}\left[(F^{m_i})_{y_iy_i}+\alpha \sigma_i\left( y_i F\right)_{y_i}
\right] \le 0
$$
everywhere in the chosen domain. The opposite sings apply to subsolutions. Similar definitions  hold    for the parabolic problem.

\medskip

We need the following technical lemma (see \cite[Lemma 2.2]{SJ05} for the proof):

\begin{lemma}
\label{Lemma Song}
Let $\alpha>0$ and $\vartheta_i>0$ for all $i=1,\cdots,N$ such that $\sum\left(\vartheta_i\alpha\right)^{-1}<1$. Then the function
$$\Upsilon(y)= \left( \sum_{i=1}^N \gamma_i|y_i|^{\vartheta_i} \right)^{-\alpha}$$
belongs to $L^1(\mathbb{R}^N\setminus B_R(0))$ for every $R>0$.
\end{lemma}

\noindent {\sl Proof of Proposition \ref{P11}}.
 We observe that Lemma \ref{Lemma Song} guarantees the summability of $\overline{F}$ outside any ball centered at the origin.

 Denoting $X=\sum_{j=1}^N\gamma_j|y_j|^{2/(1-m_j)}$, for $y\in\mathbb{R}^N\setminus\cup_{i=1}^N\{y\in \mathbb{R}^N:y_i=0\}$  and stressing that $2/(1-m_i)\geq2$  we have
\begin{equation*}
\begin{split}
I:=&\sum_{i=1}^{N}\left[(\overline{F}^{m_i})_{y_iy_i}+\alpha\sigma_i\left( y_i \overline{F}\right)_{y_i}
\right]
\\
&\leq
\sum_{i=1}^{N} 4 m_i(m_i+1)\left(\frac{\gamma_i}{1-m_i}\right)^{2}X^{-(m_i+2)}|y_i|^{2\frac{1+m_i}{1-m_i}}
+\alpha X^{-1}-2\alpha X^{-2}\sum_{i=1}^{N} \frac{\sigma_i\gamma_i}{1-m_i}|y_i|^{\frac{2}{1-m_i}}
\\
&=X^{-1}\left[
\sum_{i=1}^{N}  4m_i(m_i+1)\left(\frac{\gamma_i}{1-m_i}\right)^{2}X^{-(m_i+1)}|y_i|^{2\frac{1+m_i}{1-m_i}}
+\alpha -2\alpha X^{-1}\sum_{i=1}^{N}\frac{ \sigma_i\gamma_i}{1-m_i}|y_i|^{\frac{2}{1-m_i}}
\right]
\\
&\leq
X^{-1}\left[
\sum_{i=1}^{N} 4 m_i(m_i+1)\left(\frac{\gamma_i}{1-m_i}\right)^{2}X^{-(m_i+1)}|y_i|^{2\frac{1+m_i}{1-m_i}}
+\alpha\left(1-\min_i\{\sigma_i\frac{2}{1-m_i}\}\right)
\right]
\end{split}
\end{equation*}
Since for every $i$ we have $$\gamma_i|y_i|^{2/(1-m_i)}\leq\sum_{j=1}^N\gamma_j|y_j|^{2/(1-m_j)}=X,$$ it follows that $$
|y_i|^{2(1+m_i)/(1-m_i)}\leq
X^{(m_i+1)}\gamma_i^{-(m_i+1)},$$ then
\begin{equation*}
I
\leq
X^{-1}
\sum_{i=1}^{N} \left[4m_i(m_i+1)\left(\frac{\gamma_i}{1-m_i}\right)^{2}\gamma_i^{-(m_i+1)}
+\frac{\alpha}{N}\left(1-\min_i\{\sigma_i\frac{2}{1-m_i}\}\right)\right],
\end{equation*}
where $1-\min_i\{\sigma_i\frac{2}{1-m_i}\}<0$ by \eqref{33}. In order to conclude that $I\leq0$ it is enough to show that
\begin{equation*}
4m_i(m_i+1)\left(\frac{\gamma_i}{1-m_i}\right)^{2}\gamma_i^{-(m_i+1)}
+\frac{\alpha}{N}\left(1-\min_i\{\sigma_i\frac{2}{1-m_i}\}\right)\leq0
\end{equation*}
for every $i=1,..,N$, i.e. \eqref{gammai}. It is easy to check that computations works  for $y\in\mathbb{R}^N\setminus\left\{0\right\}$. Finally, we stress that
$({\overline{F}^{m_i}})_{y_i}\in L^{2}(\mathbb{R}^N\setminus B_R(0))$ with $R>0$
and then we can easily conclude that $\overline{F}$ is a weak supersolution as well.
\qed

\medskip


\noindent $\bullet$ {\bf Scaling for $\bar F$}. The change of mass \eqref{Tk} gives
\begin{equation}\label{Change mass}
{{
\int_{\mathbb{R}^N\setminus {\mathcal T}_k B_r(0)} F_k(y)dy= k\int_{\mathbb{R}^N\setminus {\mathcal T}_k B_r(0)} F(y_i\,k^{\nu_i})\,dy=k^{\beta} \,\int_{\mathbb{R}^N\setminus B_r(0)} F(z)\,dz\,
}}
\end{equation}
for any $r>0$,  where $\beta$ is given by \eqref{beta}   and
$${\mathcal T}_k B_r(0)=
\{y:(k^{\nu_1}y_1,\cdots,k^{\nu_N}y_N)\in B_r(0)\}.
 $$
 We will replace $\overline{F}$ with the rescaled version ${\mathcal T}_k \overline{F}$ with some large $k$,  and both the mass and the maximum of  ${\mathcal T}_k \overline{F}$ will grow  according to the rates $k^\beta$ and $k$ respectively. Moreover, it is easy to check that the following property holds:

\begin{lemma}\label{monot TK}If $0<k_1<k_2$, then ${\mathcal T}_{k_1}\overline{F}(y)<{\mathcal T}_{k_2}\overline{F}(y)$ in the common domain, where $\overline{F}$ is given by \eqref{outer.barr bis}.
\end{lemma}

\begin{remark}
We stress that  we can replace the ball $B_r(0)$ with an  \textquotedblleft anisotropic ball\textquotedblright $\,\,\widetilde{B}_r(0)=\{y:\sum_{i=1}^N\gamma_i|y_i|^{\frac{2}{1-m_i}}<r\}$, obtaining that $\overline{F}$ is a weak supersolution to \eqref{StatEq} in $\mathbb{R}^N\setminus {\widetilde{B}_r(0)}$.  

\end{remark}

\subsection{Upper comparison}\label{ssec.uc}

We are ready to prove a comparison theorem that is needed in the proof of existence of the self-similar fundamental solution.  We rely on the construction of a suitable upper barrier.
We take as first candidate a suitable rescaled  version of the function $\overline{F}(y)$ given in \eqref{outer.barr bis}
according to formula  \eqref{Tk}, i.e., $\mathcal{T}_k\overline{F}(y)$. For fixed $r>0$ and $k>0$ we denote the rescaled set of $\widetilde{B}_r(0)=\{y:\sum_{i=1}^N\gamma_i|y_i|^{\frac{2}{1-m_i}}<r\}$ by
\begin{equation}\label{Tk Omega}
\mathcal{T}_k\widetilde{B}_r(0)=\left\{y \in\mathbb{R}^N:(k^{\gamma_1}y_1,\cdots,k^{\gamma_N}y_N )\in \widetilde{B}_r(0) \right\}.
\end{equation}
In order to have a bounded global barrier $G_k$,
we define $G_k$ in $\mathcal{T}_k\widetilde{B}_r(0)$ as a constant equal to the value
$\mathcal{T}_k \overline{F}$ takes at the boundary of $\mathcal{T}_k\widetilde{B}_r(0)$. Thus, for fixed $r>0$ and $k>0$ we define
\begin{equation}\label{G_k}
G_{k}(y)=\min\{
\mathcal{T}_{k}\overline{F}(y),
\max_{\mathbb{R}^N\setminus\mathcal{T}_{k}\widetilde{B}_r(0)}\mathcal{T}_{k}\overline{F}(y)\}.
\end{equation}
Recall that $\mathcal{T}_k$ is defined in \eqref{Tk} and $\overline{F}$ is given in \eqref{outer.barr bis}.

The following comparison result is stated in terms of the solutions $v$ of rescaled equation \eqref{APMs}. We recall that the relation between $u$ and $v$ is given by \eqref{NewVariables} and the equation is invariant under time shift $t_0$. We also recall that $\tau_0=\log t_0$  (for every $t_0\in\mathbb{R}$) is the initial time for the Cauchy problem for \eqref{APMs}, \textit{i.e.} $v(y,\tau_0)=v_0(y)$.

\begin{theorem}[Barrier comparison]\label{thm.barr} Let $v$ be a solution of \eqref{APMs} with a nonnegative initial datum $v(y,\tau_0)=v_0(y)\in L^1(\mathbb{R}^N)\cap L^\infty(\mathbb{R}^N)$ such that
\begin{itemize}
\item[(i)] $v_0(y)\leq L_1$ a.e. in $\mathbb{R}^N$
\item[(ii)]$\int v_0(y) \,dy\leq M$,
\end{itemize}
where  $M>0$ and $L_1>0$. Then there exists $k$ large enough such that
$$
v_0(y)\leq G_k(y) \qquad \mbox{ a.e in } \mathbb{R}^N\setminus\mathcal{T}_{k}\widetilde{B}_r(0)
$$
implies
\begin{equation}\label{outer.comp}
v(y,\tau)\le G_k(y)  \quad \mbox{ for } \ y\in \ren, \ \tau>\tau_0,
\end{equation}
\end{theorem}

\begin{figure}[t!]
 \centering
 \vspace{-3.0cm}
\includegraphics[scale=1.3]{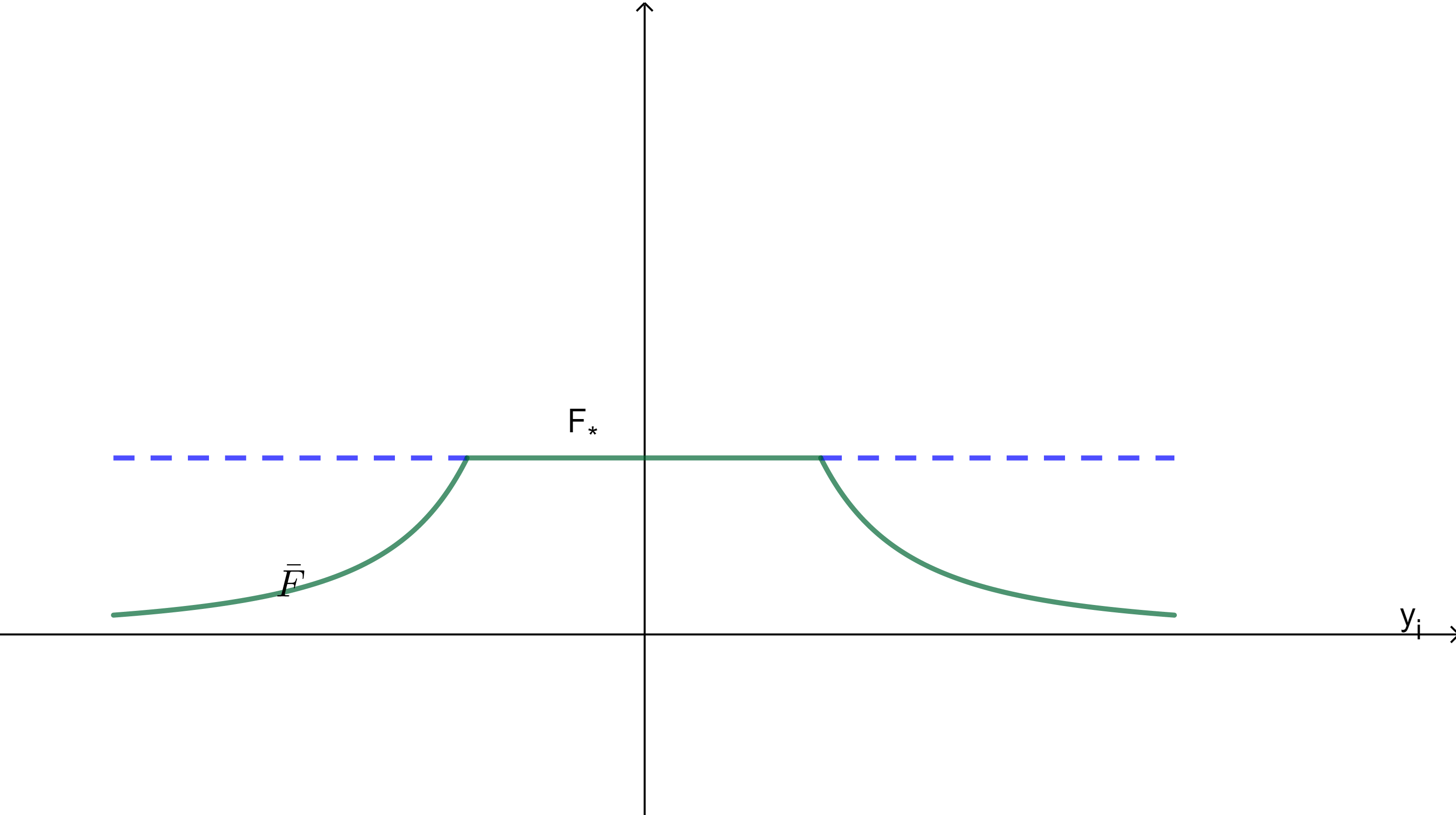}
 \caption{The barrier for $y_j=0,\forall j\neq i$ and $k=1$}
  \label{fig:boat2}
\end{figure}

\noindent {\sl Proof.} (a) Without loss of generality we fix $t_0=1$ and then $\tau_0=0$. Let us pick some $\tau_1>0$ to consider first the time $\tau \geq \tau_1$ and later the interval $[0,\tau_1)$. We denote by $F_*=\max\{\mathcal{T}_k\overline{F}(y): y\in \mathbb{R}^N\setminus\mathcal{T}_k\widetilde{B}_r(0)\}$ and we choose $k\geq 1$ such that
 \begin{equation}\label{scelta k}
\max\left\{L_1 e^{\alpha\tau_1}, C M^{2\alpha/N}(1-e^{-\tau_1})^{-\alpha}\right\}\leq F_*\equiv {{ \max_{\mathbb{R}^N\setminus \mathcal{T}_{k}\widetilde{B}_r(0)}\mathcal{T}_{k}\overline{F}(y)}}.
\end{equation}
Using the smoothing effect \eqref{Linfty-L1} and the scaling transformation \eqref{NewVariables}, we get that
\begin{equation}\label{77}
v(y,\tau)= (t+1)^{\alpha}u(x,t)\le C M^{2\alpha/N}((t+1)/t)^{\alpha}= C M^{2\alpha/N}(1-e^{-\tau})^{-\alpha},
\end{equation}
where $C$ is the constant that appears in \eqref{Linfty-L1}. By \eqref{scelta k} we have $\|v(\tau)\|_\infty\le F_*$ for all $\tau\ge \tau_1$.

(b) For $0\leq\tau<\tau_1 $ we argue as follows: from $v_0(y)\le L_1$  we get $u_0(x)\le L_1$, so by Proposition \ref{decay of the $L^p$} we have  $u(x,t)\le L_1$, therefore by \eqref{scelta k}
$$
\|v(\tau)\|_\infty \le L_1 (t+1)^{\alpha}= L_1e^{\alpha\tau}\leq F_{\ast}.
$$

(c) Under this choice we get $\|v(\tau)\|_\infty\le F_*$  for every $\tau>0$, which gives a comparison between  $v(y,\tau) $ with $G_k(y)$ in the interior cylinder $Q_{int}:={{\mathcal{T}_{k}\widetilde{B}_r(0)}}\times (0,\infty)$. In the outer cylinder $Q_{ext}:= (\mathbb{R}^N \setminus \mathcal{T}_{k}\widetilde{B}_r(0)) \times (0,\infty)$ we use the comparison principle for the $v$ variable as in Lemma \ref{comparisboundedunbound} which applies for solutions and supersolutions defined in $Q_{ext}$ and ordered on the parabolic boundary, which consists of the initial time border and the lateral border. We conclude that
$$
v(y,\tau)\leq {{\mathcal{T}_k\overline{F}}(y)} \quad \mbox{ for } \ y\in {{\mathbb{R}^N\setminus\mathcal{T}_{k}\widetilde{B}_r(0)}}, \ \tau>0,
$$
using Lemma \ref{monot TK}. The comparison for $y\in{{\mathcal{T}_{k}\widetilde{B}_r(0)}}$ has been already proved, hence the result \eqref{outer.comp}. \qed

\begin{remark}\label{remark reg}
Note that when $v_{0}\in C_{c}^{\infty}(\R^{N})$ there exists an integral bounded barrier  depending only on $L_1$ and $M$. The existence of some integrable barrier is essential to prove that the solution constructed in Section \ref{ssec.approx} is in $C([0,\infty): L^1(\mathbb{R}^N))$, see an instance in the proof of \cite[Theorem 1.2]{SJ06}.  See also subsequent sections here, where the accurateness of the asymptotic behavior as $|x|\to \infty $ plays an important role.
\end{remark}


\section{Aleksandrov's reflection principle}\label{ALEKSAND}

This is an auxiliary section used in the proof of Aleksandrov's symmetry principle so we will skip unneeded generality.
Let $H_j^+=\{x\in \mathbb{R}^N:x_j>0\}$  be the positive half-space with respect to the $x_j$ coordinate for any fixed $j\in\{1,\cdots,N\}$.  Its boundary is the hyperplane $H_j=\{x_j=0\}$  that divides $\mathbb{R}^N$ into the half spaces $H_j^+=\{x_j>0\}$ and $H_j^-=\{x_j<0\}$. We denote by $\pi_{H_j}$ the specular symmetry that maps a point $x\in H_j^+$  into $\hat x= \pi_{H_j}(x)\in H_j^-$, its symmetric image with respect to $H_j$, where $\widetilde x_j= -x_j$, $\widetilde x_i=x_i$ for $i\ne j$. We have the following important result:

\begin{proposition}\label{Prop Al}
Let $u$  be   a  nonnegative solution of the Cauchy problem for \eqref{APM} with
nonnegative initial data $u_0\in L^1(\mathbb{R}^N)$.
If  for a given hyperplane $H_j$ with  $j=1,\cdots,N$ we have
$$u_0(\pi_{H_j}(x))\leq u_0(x)\, \text{ for all }x\in H{_j}^{+}$$
then 
$$
u(\pi_{H_j}(x),t)\leq u(x,t)\quad \text{ for all }(x,t)\in H_{j}^{+} \times  (0,\infty).
$$
\end{proposition}

\begin{proof}
We first observe that if $u(x,t)$ is a solution to Cauchy problem with initial datum $u_0(x)$, then $\widetilde{u}(x,t)=u(\pi_{H_j}(x),t)$ is a solution to Cauchy problem with initial datum $u_0(\pi_{H_j}(x))$.
 By approximation we may assume that the solutions are
continuous and even smooth, and continuous at $t = 0$ as explained in Subsection \ref{ssec.approx}.
We consider in $Q^+ =H_j^+\times(0,+\infty)$ the
solution $u_1(x,t) = u(x,t)$ and a second solution
$u_2(x, t) = u(\pi_{H_j}(x), t)$. Our aim is to show that
\[
u_2(x,t)\leq u_1(x,t)\quad \text{ for all }(x,t)\in H_{j}^{+} \times  (0,\infty).
\]
By assumption the initial values satisfy $u_2(x, 0)\leq u_1(x, 0)$  and the boundary values on $\partial H_j^+\times[0,+\infty)$ are the same. Then Lemma \ref{comparisboundedunbound} yields the assertion.
\end{proof}

We now introduce the concept of \sl separately symmetric and nonincreasing function\rm, SSNI. Precisely, a function $g:\mathbb{R}^N\rightarrow\mathbb{R}$ is \ SSNI if it is a  symmetric function in each variable $x_i$ and a nonincreasing function in $|x_i|$ for all $i$,\textit{ i.e.}
\begin{equation}\label{simm}
 g(x_1,\cdots,x_N)=g(|x_1|,\cdots,|x_N|) \quad \forall x \in \mathbb{R}^N,
\end{equation}
and for all $j=1,\cdots,N$
\begin{equation}\label{mon}
g(|x_1|,\cdots,|x_j| ,\cdots,|x_N|)\leq g(|x_1|,\cdots,|\widehat{x}_j| ,\cdots,|x_N|) \quad \text{ if }|\widehat{x}_j|\leq |x_j|.
\end{equation}
We say that the evolution function $u(x,t)$ is  \ SNNI if it is an \ SSNI function with respect to the space variable for all $t>0$.
The next result states the conservation in time of the SSNI property under the AFDE flow.

\begin{proposition}\label{Prop 3}
Let $u$ be a  nonnegative solution of the Cauchy problem for \eqref{APM} with nonnegative
initial data $u_0\in L^1(\mathbb{R}^N)$. If $u_0$ is a  symmetric function in each variable $x_i$,  and also a  nonincreasing   function in $|x_i|$ for all $i$, then $u(x,t)$  is also symmetric and a nonincreasing function in $|x_i|$ for all $i$ for all $t$.
\end{proposition}
\begin{proof}
By Proposition \ref{Prop Al} the solution $u(x,t)$ is a function in $|x_i|$. Finally, Lemma \ref{comparisboundedunbound} applied in $H_i^+$, to $u(x,t)$ and to $\widehat{u}(x,t)=u(x_1,\cdots,x_i+h,\cdots,x_N,t)$ yields that $u$ is a nonincreasing function in $|x_i|$.
\end{proof}

\section{A quantitative positivity lemma}\label{sec.qual}

 As a consequence of mass conservation and the existence of the upper barrier we obtain a positivity lemma for certain solutions of the equation. This is the uniform positivity that is needed in the proof of existence of self-similar solutions, it avoids the  fixed point from being trivial. A similar but simpler barrier construction was done in  \cite{Vaz20} where radial symmetry was available.
 It is convenient to use the rescaled variable $v$ of \eqref{NewVariables} instead of $u$.

\begin{lemma}\label{lem.pos}
Let $v$ be the solution of the rescaled equation \eqref{APMs} with a nonnegative SSNI initial datum $v_0\in L^\infty(\ren)\cap L^1(\ren)$ with mass $M>0$, such that $v_{0}\leq  G_k$ a.e. in
${ {\mathbb{R}^N\setminus\mathcal{T}_k \widetilde{B}_r(0)}}$, where $G_{k}$ is a suitable barrier defined as in \eqref{G_k} and $\mathcal{T}_k  {\widetilde{B}_r(0)}$ is defined in \eqref{Tk Omega}. Then, there is a continuous nonnegative function $\zeta(y)$, positive in a ball of radius $r_0>0$, such that
\begin{equation*}
v(y,\tau)\ge  \zeta(y) \quad \mbox{ for all } \ y\in \ren, \ \tau>0.
\end{equation*}
In particular, we may take $\zeta(y)\ge c_1>0$ in $ B_{r_0}(0)$ for suitable $r_0$ and $c_1>0$, to be computed  below.
\end{lemma}

\noindent {\sl Proof.} We know that for every $\tau>0$ the solution $v(\cdot,t)$ will be nonnegative, and also by the previous section it is SSNI. By Theorem \ref{thm.barr} there is an upper barrier $G(y)=G_k(y)$  for $v(y,\tau)$ for every $\tau$.
 We stress that we need $v_0$ to be below a suitable barrier to use Theorem \ref{thm.barr}.
 Since $G$ is integrable, there is a large box $Q=\{y: |y_i|\le R\}$  such that in the outer set  $O=\mathbb{R}^n\setminus Q$ we have a small mass:
$$
\int_{O} v(y,\tau)\,dy\le
\int_{O} G(y)\,dy< M/5,
$$
for all $\tau>0$. On the other hand, we consider the set $A_i(r_0)=\{y\in Q:| y_i| \le r_0\}$. Since for small $r_0$ this set has a volume of the order of $r_0R^{N-1}$ and the function $G$ is bounded by a constant $C_1$ we have
$$
\int_{A_i(r_0)} v(y,\tau)\,dy\le
\int_{A_i(r_0)} G(y)\,dy\le C_1R^{N-1}r_0
$$
for all $\tau>0$. By choosing $r_0>0$ small we can get this quantity to be less than $M/4N$. This calculation is the same for all $i$. 

Now we look at the integral in the complement of the above sets, i.e., the remaining set $D=\{y: r_0<|y_i|<{R} \ \mbox{\rm for all }\ i\}$. Note that this set is composed of  $2^N$ symmetrical copies (see Figure \ref{fig:boat1} for the two dimensional case). Using the mass conservation we get
$$
\int_{D} v(y,\tau)\,dy\ge M-{ M/5}  -M/4 >M/2.
$$
Since $v$ is an SSNI function, we get in each of the $2^N$ copies the same result, so if
$D^+=\{y\in D: y_i>0 \ \mbox{\rm for all } \ i\} $ {(see Figure \eqref{fig:boat1})}  we get
$$
\int_{{D^+}} v(y,\tau)\,dy >M/2^{N+1}.
$$

\begin{figure}[t!]
 \centering
 \vspace{-3.0cm}
 \includegraphics[scale=0.1]{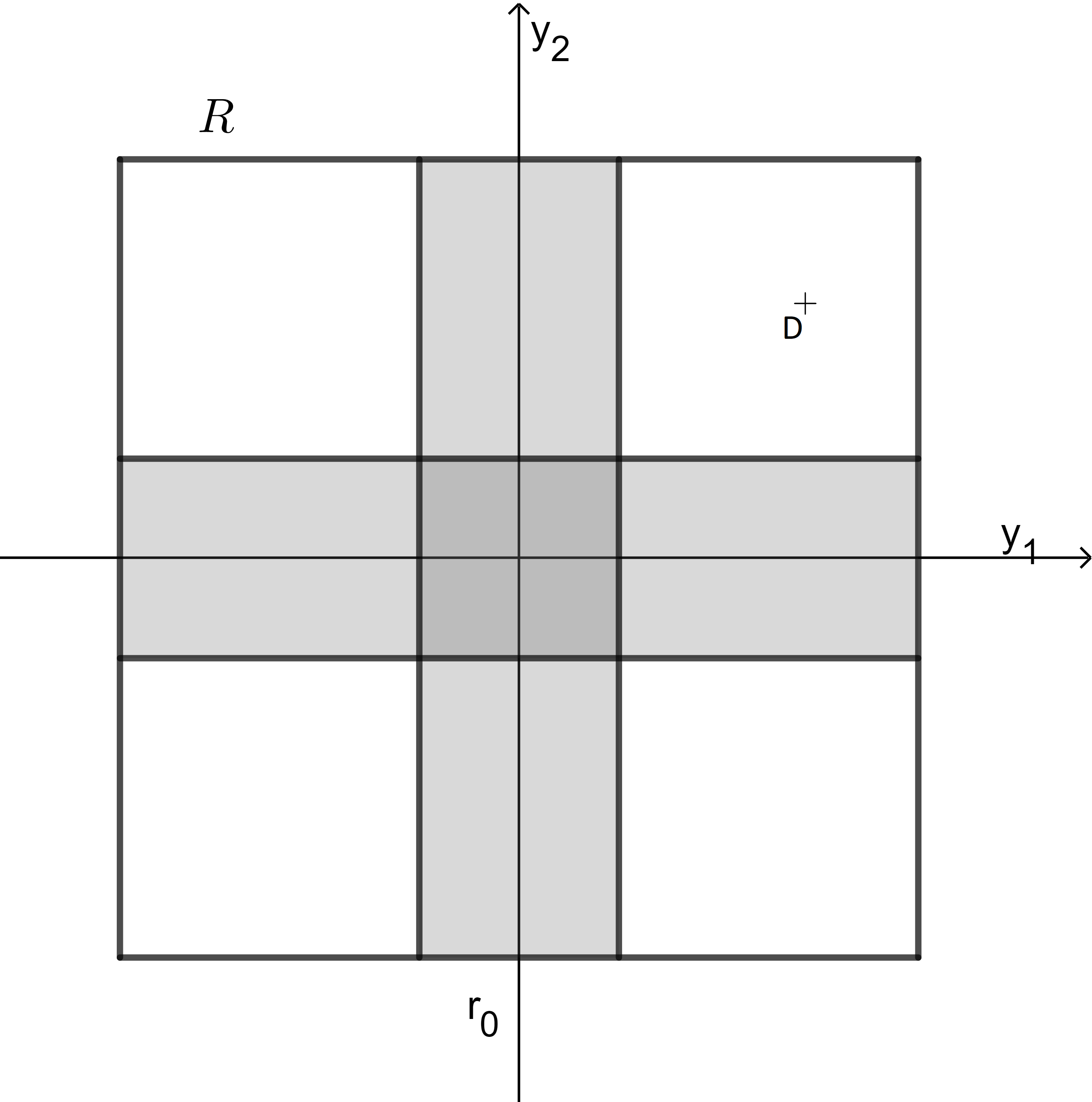}
 \caption{the set $D$ made by the union of the four white cubes and its subset$D^+$}
  \label{fig:boat1}
\end{figure}

Now we use the monotonicity in all variables to show that at the bottom-left corner point of $D^+$ we obtain a maximum, hence
$$
v(r_0,r_0,...,r_0,\tau)( R  -r_0)^N\ge M/2^{N+1}.
$$
Using again the separate monotonicity and symmetry of $v(\cdot,\tau)$, we conclude that
 $$
 v(y,\tau)\ge c_1 \qquad \mbox{ for all} \ |{y_i}|\le r_0, \ \tau>0,
 $$
 with $c_1=M\,2^{-(N+1)}( R -r_0)^{-N}$. The function $\zeta$ can be constructed as cut-off function, whose value is $c_{1}$ in the ball $B_{r_0/2}(0)$ and vanishing outside $B_{r_0}(0)$.
 This concludes the proof.
\qed


\section{Self-similar fundamental solution. Existence, uniqueness and properties}\label{sec.ex.ssfs}

Here we  prove one of the main results of  the present paper, Theorem \ref{fundamental solution}, concerning the existence of a unique fundamental solution  $U_M(x,t)$ of  the self-similar type \eqref{sss} with mass $M$, $M>0$.
 We start by the uniqueness part, Subsection \ref{sec.uniq}. The existence part will be discussed in Subsection \ref{existself}.
Precise monotonicity,  positivity and regularity properties occupy the final subsections.


\subsection{Proof of uniqueness of the  self-similar fundamental solution}\label{sec.uniq}

We settle the issue regarding the uniqueness of the  self-similar fundamental  solution  stated in Theorem \ref{fundamental solution}. We recall that the  profile $F_{M}$ of a self-similar fundamental  solution $U_M$ is nonnegative, bounded and  goes to zero as $|y|\to\infty$ with a certain multi-power rate.

   The proof combines a number of different arguments.

$\bullet$  \textit{Step 1.a. Aleksandrov's principle.}   We first recall the following anisotropic version of the monotonicity result \cite[Proposition 14.27]{Vlibro}. It uses another version of Aleksandrov's principle.

\begin{lemma}\label{monotonicitProp}
Let $u\geq0$ be a solution of the Cauchy problem for equation \eqref{APM} with initial data supported in a box $Q_{R}=[-R,R]^{N}$, $R>0$. Then for every $x,\widetilde{x}\in \ren$,  $x,\widetilde{x}\not\in Q_{2R}$   such that for some fixed $i$, the $i$-th coordinates  satisfy \, $|\widetilde{x}_i|\ge |x_i|+2R$, while the other components, labeled $j$, are equal, \textit{i.e.} $\widetilde{x}_j=x_j$ for all $j\neq i$. Then we have
\begin{equation}\label{alek.ssni}
u(x,t)\ge u(\widetilde{x},t).
\end{equation}
\end{lemma}
\begin{proof}
We consider the hyperplane $H$ which is the mediatrix between the points $x$ and $\widetilde{x}$ and let  $H^+$ and $H^-$ be the two half-spaces such that $x\in H^+$ and $\widetilde{x}\in H^-$. We denote by $\pi_H$ the specular symmetry that maps a point $x\in H^+$ into a point $\pi_H(x)\in H^-$  as in Proposition \ref{Prop Al}.  Let $u_1(x,t)=u(x,t)$ in $H^+\times (0,+\infty)$  and $u_2(x,t)=u_1(\pi_H(x),t)$.  By our choice of $\widetilde{x}$ and $x$ the initial value of $u_2$  at $t=0$ is zero because the support of $u(\cdot,0)$ is inside $Q_{R}$ while the initial value of $u_1$ is nonnegative in $H^+$. Moreover the values on $H\times(0,+\infty)$ and on $\partial Q_{2R}\times(0,+\infty)$ are the same by construction.  Then by Lemma \ref{comparisboundedunbound} with $\Omega=H^+\setminus Q_{2R}$ we get \eqref{alek.ssni}.
\end{proof}

   $\bullet$ \textit{Step 1.b. SSNI property  of the self-similar solutions}. Now we can prove the following property of the self-similar solutions.

\begin{lemma}\label{prop.ssni} Any non-negative self-similar solution $U_M(x,t)$ with finite mass $M$ is SSNI.
\end{lemma}

\noindent {\sl Proof}. We use two general ideas:

i) SSNI  is an asymptotic property of many solutions,

ii) self-similar solutions necessarily
verify asymptotic properties for all times.

 Let  us consider a non-negative self-similar solution, $U_M(x,t)$, of the self-similar form $U_M(x,t)=t^{-\alpha}F_M(t^{-a_1}x_1,..,t^{-a_N}x_N)$, see \eqref{sss}. We must prove that it fulfills the SSNI property. This is done by approximation, rescaling, and passing to the limit. We start by approximating $U_M$ at time $t=1$ with a sequence of bounded, compactly supported functions $u_n(x,1)$ with increasing supports and converging to $U_M(x,1)$ in $L^1(\mathbb{R}^N)$. We consider the corresponding solutions $u_n(x,t)$ to \eqref{APM}, for times $t\ge 1$.

 (1) Because of their compact support at $t=1$, the Aleksandrov principle implies that these functions $u_n(\cdot,t)$ satisfy for all $t>1$ an approximate version of the SSNI properties as follows. If the initial support at $t=1$ is contained in the box $Q_{R_{n }}=[-R_{ n },R_{n }]^N$, $R_{ n }>0$, then by Lemma \ref{monotonicitProp}, we have that for all $t>1$ and for every $x,\widetilde{x}\in \ren$,  $x,\widetilde{x}\not\in Q_{2R_{n }}$, it holds
\begin{equation}\label{alek.ssni}
u_n(x,t)\ge u_n(\widetilde{x},t)
\end{equation}
on the condition that  for every fixed $i$  $|\widetilde{x}^i|\ge |x^i|+2R_n$ and $\widetilde{x}^j=x^j$ for all $j\neq i$. The length $2R_{ n }$ plays here the role of error in the SSNI property in what regards monotonicity in every coordinate direction.

(2) Since the self-similar solution has typical penetration length  $x_i\sim t^{a_i}$ in each direction and $a_i>0$,  the previously detected error length $2R_n$ becomes comparatively negligible.  It is now convenient to pass to  rescaled variables as in \eqref{NewVariables} (with $t_0=0$, so that $t=1$ means $\tau=0$, and  $y_i=x_i\,t^{-a_i}$). Then, the monotonicity properties, as just derived for $u_n$ by Aleksandrov, keep being valid in terms of $(y_1,\cdots, y_N)$ with the reformulation:
\begin{equation}\label{ssni}
v_{n}(y,\tau)\ge v_{n}(\widetilde{y},\tau)
\end{equation}
 for $y,\widetilde{y}\not\in \prod_{i=1}^N [-R_n t^{-a_i},R_nt^{-a_i}]$ on the condition that for every fixed $i$ \ $|\widetilde{y}_i|\ge |y_i|+2R_n t^{-a_i}$ and $\widetilde{y}_j=y_j$ for all $j\neq i$ ,  $\tau= \log(t)$. Similarly, symmetry comparisons are true up to a displacement $R_n\,t^{-a_i}$. We also note that by the contraction principle, for $\tau \geq 0$
and $n\ge n(\ve)$ we have after an easy computation
\[
\|v_{n}(\cdot, \tau)-F_M\|_{L^{1}(\R^{N})}\leq \|u_{n}(1)-U_M(1)\|_{L^{1}(\R^{N})}\le \ve.
\]

(3) Now we pass to the limit in $n,\tau$ and $\ve$ to translate the asymptotic approximate properties into
exact properties. We first let $\tau\to\infty$ with  $\ve$ and  $n\ge n(\ve)$ fixed.
We observe that $v_n(y,\tau)$  converges to some $V_n(y)$ along some subsequence $\tau_k\to\infty$,  using the smoothing effect \eqref{Linfty-L1}.  The necessary compactness of the orbit
$v_{n}(\cdot, \tau)$ follows from the energy estimates  \eqref{Energywholespace} for equation \eqref{APM} after translating them to the $v$-equation, an operation that keeps the estimates uniformly bounded. Compactness in time follows from the Aubin-Lions-Simon lemma, used as in Subsection 2.

 We stress that $R_nt_k^{-a_i}:=R_n e^{-\tau_k a_i}\rightarrow0$ as $\tau_k\to\infty$. From
\eqref{ssni} we get
\begin{equation}
V_{n}(y)\ge V_n(\widetilde{y})
\end{equation}
on the condition that for every fixed i $|\widetilde{y}_i|\ge |y_i|$ and $\widetilde{y}_j=y_j$ for all $j\neq i$.  We also have    by Fatou's lemma that   $\|V_{n}-F_M\|_{L^{1}(\R^{N})}\leq \ve$. This implies after letting $\ve\to 0$ (hence, $n(\ve)\to\infty$) that
\[
F_M(y)\geq F_M(\widetilde{y})
\]
for {almost every $y$, $\widetilde{y}$} such that for every fixed $i$,  $|\widetilde{y}_i|\ge |y_i|$ and $\widetilde{y}_j=y_j$ for all $j\neq i$. {Therefore, the continuity of $F_{M}$ implies that such monotonicity is preserved for \emph{all} the points $y$, $\widetilde{y}$ with the previous property}.

 By a similar argument, $F_M$ is symmetric with respect to each $ y_{i}$ and the full SSNI applies to $F$, hence to the original $U_M$. \qed

\medskip

 In the proof of the uniqueness, we need the following lemma on the set of positivity of profile $F_M$.

\begin{lemma}\label{star shap} The set $\Omega_0=\left\{y\in \mathbb{R}^N: F_M(y)>0\right\}$ is star-shaped from the origin,\textit{ i.e.,} for all $y^0\in\Omega_0$ the line segment from $0$ to $y^0$ lies in $\Omega_0$.
\end{lemma}
\noindent {\sl Proof.} We stress that $F_M(0)>0$, then $0\in \Omega_0$. Let us take $y^0\in\Omega_0$ and consider the segment $y_i=\frac{y_i^0}{s_0}s$ for $i=1,\cdots,N$ and $s\in[0,s_0]$. Recalling that $F_M$ is SSNI (see Lemma \ref{prop.ssni}), then $F_M(y)= F_M(\frac{|y_1^0|}{s_0}s, \cdots,\frac{|y_N^0|}{s_0}s)\geq F_M(\frac{|y_1^0|}{s_0}s_0, \cdots,\frac{|y_N^0|}{s_0}s_0) =F(y^0)>0$. \qed

\medskip

\noindent    $\bullet$ \textit{Step 2. Mass Difference Analysis. }  We enter the main steps in the proof of uniqueness. The first one is the mass difference analysis.

(i) {Let us suppose that $U_1$ and $U_2$ are two self-similar fundamental solutions with the same mass $M_1=M_2=M>0$ and profiles $F_{1},\,F_{2}$. We introduce the functional}
 \begin{equation}\label{J}
 J[U_1,U_2](t)=\int_{\mathbb{R}^N} (U_1(x,t)-U_2(x,t))_+\,dx.
 \end{equation}
 By the accretivity of the operator  this is a Lyapunov functional, i.e., it is nonnegative and nonincreasing in time. Observe that we have the formula,
  \begin{equation}
 J[U_1,U_2](t)=\int_{\mathbb{R}^N} (F_1(x)-F_2(x))_+\,dx,\label{JF}
 \end{equation}
 \emph{i.e.} $J[U_1,U_2]$ must be constant in time for self-similar solutions, that is
 \begin{equation}
 J[U_1,U_2](t)=c_0 \label{constantJ}
 \end{equation}
 for a suitable constant $c_0\geq0$.

(ii) The main point is that such different solutions with the same mass must intersect.
 We argue as follows: we define at the time $t=1$, the maximum of the two profiles $G^*=\max\{F_1,F_2\}$, and the minimum $G_*=\min\{F_1,F_2\}$. Let $U^*$ and $U_*$ the corresponding solutions of   \eqref{APM} for $t>1$. By Theorem \ref{contraction} we have for every such $t>1$
  \begin{equation}\label{eq.order}
U_*(x,t)\le U_1(x,t), U_2(x,t)  \le U^*(x,t).
\end{equation}

%
%

 We claim that $U^*(x,t)$, $t>1$, is a self-similar solution that equals the maximum of the two solutions $U_1$ and $U_2$, and similarly, $U_*(x,t_1)$, $t>1$,  is a self-similar solution that equals the minimum of the two solutions.
First, note that
\begin{equation}\label{57.2}
U^*\geq \max\{U_1,U_2\} \text{ for all }(x,t)\in \mathbb{R}^N\times[1,+\infty)
\end{equation}
by Theorem \ref{contraction}. Next, by the mass preservation, we get
\begin{equation}\label{57.3}
\int_{\mathbb{R}^N}U^*(x,t)\,dx=\int_{\mathbb{R}^N}U^*(x,1)\,dx=\int_{\mathbb{R}^N}G{*}(x)\,dx
\end{equation}
for all $t\geq1$. Since
$$\max\{U_1,U_2\}(x,t)=t^{-\alpha}\max\{F_1,F_2\}(t^{-\alpha \sigma_1}x_1,\cdots,t^{-\alpha \sigma_N}x_N)$$
we have
$$\int_{\mathbb{R}^N} \max\{U_1,U_2\}(x,t)\,dx=\int_{\mathbb{R}^N} \max\{F_1,F_2\}(x)\,dx=\int_{\mathbb{R}^N} G^{*}(x)\,dx$$
for all $t>0$. Consequently,
$$\int_{\mathbb{R}^N} U^*(x,t)\,dx=\int_{\mathbb{R}^N}\max\{U_1,U_2\}(x,t)\, dx$$
for all $t\geq1$. Combining last inequality and {\eqref{57.2}} we conclude that
$$
U^*(x,t)= \max\{U_1,U_2\}(x,t)\text{ for all }(x,t)\in \mathbb{R}^N\times[1,+\infty).
$$
Similarly,
$$
U_*(x,t)= \min\{U_1,U_2\}(x,t)\text{ for all }(x,t)\in \mathbb{R}^N\times[1,+\infty).
$$
 Note that both are self-similar solutions.

\noindent $\bullet$\textit{    Step 3.  Strong Maximum Principle. } Finally, we use the strong maximum principle (SMP for short) to show that the last conclusion is impossible in our setting.

(iii) Since $ U^*(x,1)=G^*(x)$, we have that
\begin{itemize}
\item [-] $U^*(x,1)\geq U_1(x,1)$ and $U^*(x,1)\geq U_2(x,1)$ for all $x\in\mathbb{R}^N$.
\item [-] $U^*(0,1)$ equals $U_1(0,1)$ or $U_2(0,1)$
\end{itemize}
Now we make the observation that for self-similar solutions touching at $x=0$ for $t=1$ implies touching at $x=0$ for every $t>1$.

We observe that for every $t\ge 1$ both  $U_1(x,t),U_2(x,t)>0$ are strictly positive at $x=0$ because they are SSNI (see Lemma \ref{prop.ssni}) and they have positive mass. By continuity $U_1(x,t),U_2(x,t)>c>0$ in a neighbourhood $I(0)$ of 0 for all $t\geq1$, $t$ close to 1.

We stress that in the open set $I(0)$ the solutions $U_1,U_2$ and $U^*$ are positive, so that we can prove locally smoothness for them since the equation is not degenerate (see  \cite[ Theorem 6.1, Chapter V]{LSU}). As a consequence, we can apply the evolution strong maximum principle   (see \cite{Ni, PS2007})   in $I(0)\times [1,t_1]$ for  $t_1>1$, $t_1-1$ small, applying it to the ordered solutions $U^*$ and $U_1$, or to $U^*$ and $U_2$.

Suppose that $U^*$ and $U_1$ touch for $t=1$ at $x=0$, i.e.,  $U^*(0,1)=U_1(0,1)$. The SMP implies that they cannot touch again for $t>1$ at $x=0$ unless they are locally the same. However, both are self-similar so that the touching point is preserved.  Indeed, since they are self-similar, if  $U^*(0,1)=U_1(0,1)$ then $U^*(0,t)=U_1(0,t)$ for all $t>1$. 
We conclude  that
$U^*=U_1$ in the whole open set $\Omega_1=\{x: U_1(x,t_1)>0\}$ and the SMP can be applied (and it holds also on its closure by continuity). By the definition of maximum of two solutions, it means that $U_1(x,t_1)\ge U_2(x,t_1)$ in $\Omega_1$.

If $\Omega_1$ is the whole of $\mathbb{R}^N$, we arrive at the conclusion that $U_1(x,t_1)\ge U_2(x,t_1)$
everywhere. This implies by equality of mass that $U_1=U_2$ at $t=t_1$. 
In that case we must have that $c_0=0$, where $c_0$ is the constant appearing in \eqref{constantJ}, therefore $U_{1}\leq U_{2}$  for all $x$  all times. By  mass conservation  we finally have $U_{1}=U_{2}$  (for all $x$ and all $t$).  The proof on uniqueness concludes in this case.

iv) We still have to consider the case where the positivity set of $U_1$, $\Omega_1$, is not known to be $\mathbb{R}^N$.
  Lemma \ref{star shap} guarantees that the set where $U_1$ is positive is   star-shaped  from the origin. If $U^*$ and $U_1$ touch for $t=1$ at $x=0$ and $\Omega_1$ is the not whole $\mathbb{R}^N$, then for every unit vector $\vec e\in\ren$ there is a point  $x=s_0\vec e$ with $s_0>0$ that belongs to the boundary of $\Omega_{1}$ and is such that $U_1(x,t_1)>0$ if $x=s\vec e$ with $s<s_0$ and $U_1(x,t_1)=0$ if $x=s\vec e$ with $s\ge s_0$.  We conclude as in the previous analysis that $U_1(x,t_1)\ge U_2(x,t_1)$ on $\partial \Omega_1$, which means by the property SSNI applied to $U_2$ that $U_2=0$ is zero outside of $\Omega_1$.  Hence, $U_1\ge U_2$ everywhere. By equality of mass $U_1=U_2$. 

 v ) A similar argument applies when $U^*{(0,1)}=U_2{(0,1)}$. At the end in every case $U_1=U_2$ and we conclude the proof of uniqueness. \qed

\subsection{Proof of existence of a self-similar solution}\label{existself}

\noindent    $\bullet $ \textit{A remark.}   We start the existence proof   with the following remark.
 Let $\phi\geq 0$ bounded, symmetrically decreasing with respect to $x_{i}$ with total mass $M$.
We consider the solution $u_1$ (uniqueness is given by Theorem \ref{contraction}) with such initial datum, i.e. $u_1(x, 0  )=\phi$, and denote
\begin{equation}\label{uk}
u_k(x,t)=\mathcal{R}_ku_1(x,t)=k^{\alpha}u_1(k^{\sigma_1\alpha}x_1,..., k^{\sigma_N\alpha}x_N,kt)
\end{equation}
for every $k>1$, which solves the main equation \eqref{APM}. We want to let $k\rightarrow\infty$. In terms of rescaled variables \eqref{NewVariables} (with $t_0= 1 $) we have
\begin{equation*}
\begin{split}
v_k(y,\tau)&=e^{\alpha\tau }u_k(y_1e^{\alpha \sigma_1 \tau},...,y_Ne^{\alpha \sigma_N \tau},e^\tau)
\\
&=e^{\alpha \tau }k^{\alpha}
u_1(k^{\sigma_1\alpha}y_1e^{\tau\sigma_1\alpha},..., k^{\sigma_n\alpha}y_Ne^{\tau\sigma_N\alpha},ke^\tau),
\end{split}
\end{equation*}
where $t=e^{\tau}$, $\tau>0$. Put $k=e^{h}$ so that $k^{\sigma_i\alpha}e^{\tau\sigma_i\alpha}=e^{(\tau+h)\sigma_i\alpha}$. Then
\begin{equation*}
v_k(y,\tau)
=e^{(\tau+h) \alpha }
u_1\left(y_1e^{(\tau+h)\sigma_1\alpha},..., y_Ne^{(\tau+h)\sigma_N\alpha},e^{(\tau+h)}\right).
\end{equation*}
Putting $v_1(y',\tau' )=t^{\alpha}u_1(x,t)$ with $y'_i=x_it^{-\alpha \sigma_i}$, $\tau'=\log t$, then
 \begin{equation*}
v_k(y,\tau)
=
e^{(\tau+h-\tau') \alpha }
v_1(y_1e^{(\tau+h-\tau')\sigma_1\alpha},..., y_Ne^{(\tau+h-\tau')\sigma_N\alpha},\tau+h).
\end{equation*}
Setting $\tau'=\tau+h$, we get
\begin{equation}\label{vk}
v_k(y,\tau)=v_1(y,\tau+h).
\end{equation}
This means that the transformation $\mathcal{R}_k$ becomes a forward time shift in the rescaled variables that we call $\mathcal{S}_h$, with $h=\log k$.

\medskip

\noindent    $\bullet$\textit{Step 1. Fixed point argument.}   The next steps prepare the setting to obtain a fixed point of $\mathcal{S}_h$.
\medskip

(i) We let $X = L^1(\mathbb{R}^N)$. We consider an important subset of $ X$ defined as follows.

\noindent  Let us fix a large constant $L_1>0$ and consider the barrier  $G_{k(L_1)}$ as in Theorem \ref{thm.barr} (with $k$ related to $L_1$ and $M=1$). We define $\mathcal K= \mathcal K(L_1)$ as the set of all $\phi\in L_{+}^{1}(\R^{N})\cap L^{\infty}(\R^{N})$  such that: \

(a) $\int \phi(y)\,dy = 1$, \

(b) $\,\phi\,$ is SSNI (separately symmetric and nonincreasing w.r.\,to all coordinates),

(c1) $\phi$  is uniformly bounded above by $L_1$.

(c2) $\phi$  is bounded above by the fixed barrier function $G_{k(L_1)}$,  \textit{i.e.} $\phi(x)\leq G_{k(L_1)}(x)$ for all $x\in\R^{N}$.  

\medskip

\noindent We stress that we can reduce ourselves to the case of unit mass, because we can pass from any mass $M>0$ to mass $M=1$ (see Subsection \ref{sec scaling} and \eqref{Change mass}).

\noindent It is easy to see that  $\mathcal K(L_1)$ is a non-empty, convex, closed and bounded subset  of   the Banach space $X$.

(ii) Next, we prove the existence of periodic orbits. For all $\phi\in \mathcal K(L_1)$ we consider the solution $v(y,h)$ to equation \eqref{APMs} starting at $\tau = 0$ with data $v(y, 0) =\phi(y)$.
We now consider for $h>0$ the semigroup map $S_h : X \rightarrow X$ defined by ${S} _h(\phi) =v(y, h)$.
The following lemma collects the facts we need.

\begin{lemma}\label{lemma.below}
Given $h>0$ there exists a value of $L_1(h)=L_1$ such that $S_h(\mathcal K(L_1))\subset K(L_1)$.  Under such situation, for every $\phi\in \mathcal K(L_1)$
\begin{equation}\label{bdd.below}
S_{h}\phi(y)\ge \zeta_{{h}}(y)   \quad \mbox{ for } \ y\in \ren, \ {h}>0.
\end{equation}
where $\zeta_{{h}}$ is a fixed function as in Lemma \ref{lem.pos}, which depends only on $h$.
\end{lemma}

\noindent {\sl Proof.}
Fix a small $h>0$, and let $L_1=L_{1}(h)$ such that
\begin{equation}\label{L1 h}
L_{1}\geq  C   (1-e^{-h})^{-\alpha},
\end{equation}
where $ C   $ is the constant in the smoothing effect \eqref{Linfty-L1}. We take $\tau_1=h$ in the proof of Theorem \ref{thm.barr} and choose the rescaled $\mathcal{T}_k\overline{F}$ such that $F_{\ast}=F_{\ast}(h)$, the maximum of $\mathcal{T}_k\overline{F}$ on  the complement of ${\mathcal{T}_k\widetilde{B}_r(0)}$, fulfills
\[
L_1 e^{\alpha h}\le F_*.
\]
Then  using \eqref{L1 h}   we have
\[
 C    (1-e^{-h})^{-\alpha}\le F_*,
\]
whence \eqref{scelta k} is satisfied.
This ensures the existence of a barrier  $G_{k(h)}(y)$   (defined in \eqref{G_k}), such that for $\phi\in \mathcal{K}(L_1(h))$ and any $\tau>0$ we have $S_{ \tau}(\phi)\leq G_{k(h)}(y)$
a.e.. Then $S_{h}(\phi)$ obviously verifies (c2), while (a) is a consequence of mass conservation and (b) follows by Proposition \ref{Prop
3}. Moreover, \eqref{L1 h} ensures that from \eqref{77} we immediately find $S_{h}(\phi)\le L_1 $ a.e., that is property (c1).
The last estimate \eqref{bdd.below} comes from
Lemma \ref{lem.pos}, which holds once a fixed barrier is determined.  \qed

\begin{lemma}\label{lemma.precomp}
The image set \ $Y=S_h(\mathcal K(L_1))$  is relatively compact in $X$.
\end{lemma}

\begin{proof} The  image set $S_h(\mathcal K)$ is bounded in $L^1(\ren)$ and $L^\infty(\ren)$ by already established estimates using the definition of $v$ in terms of $u$. The Fr\'echet-Kolmogorov theorem says that a subset $Y$ is relatively compact in $L^1(\ren)$ if and only if the following two conditions hold

(A) (Equi-continuity in $L^1$ norm)
$$
\displaystyle \lim_{| z | \to 0} \int_{\ren} |f(y)-f (y+z)|\,dy = 0
 $$
 and the limit is uniform on $f\in Y$.

(B) (Equi-tightness)
$$
\displaystyle \lim _{R\to \infty }\int _{|x|>R}\left|f(y)\right|\,dy=0
$$
and the limit is uniform  on $f\in Y$.

In our case the second property comes from the uniform upper bound by a common function $G_{k(h)}$. So for every $\ve>0$ we can find $R(\ve)>0$ such that
$\int _{|x|>R(\ve)}\left|f\right|\,dy<\ve$ for all $f=S_t\phi$ and all $t>0$.

For the proof of (A) we proceed as follows. Let $v(\cdot,t)= S_t\phi$. As a consequence of the energy estimates {{\eqref{Energywholespace}}}, all the derivatives $\partial_i v^{m_i}$ are bounded in $L^2(0,T:L^2(\ren))$. Since $\partial_i v= (1/m_i) v ^{1-m_i}\partial_i v^{m_i}$ and $v$ is bounded in this time interval we conclude that $|\nabla v|$ is bounded in $L^2(0,h:L^2(\ren))$. This means that for some $\tau_0 \in (h/2,h)$ the integral
$$
\int_{\ren} |\nabla v(y,\tau_0)|^2\,dy \le \frac2{h}\int_{h/2}^h\int_{\ren} |\nabla v(y,\tau)|^2\,dyd\tau \le C_2/h,
$$
where $C_2$ depends only on $L_1$. By an easy functional immersion this implies that  for every small displacement  $z$ with $|z|\le \delta$ we have and for every $R>0$
$$
\int_{B_R(0)} |v(y,\tau_0)-v(y+s,\tau_0)|\,dy \le \delta C_3
$$
and $C_3$ is a constant that depends only $R$, $h$ and on $C_2$. This equi-continuity bound in the interior is independent of the particular initial data in $\phi  \in \mathcal K$. Putting $R=R(\ve)$ and using the uniform bound $S_t \phi\le G_{k(h)}$ we get full equi-continuity at $\tau=\tau_0$:
$$
\int_{\ren} |v(y,\tau_0)-v(y+z,\tau_0)|\,dy \le 2\ve
$$
uniformly on $\phi\in \mathcal K$ if $\delta$ is small enough. Since both $v(y,\tau)$ and $v(y+z,\tau)$ are solutions of the renormalized equation, we conclude from the $L^1$ contraction property \eqref{contraction} that
$$
\int_{\ren} |v(y,\tau_0)-v(y+z,\tau_0)|\,dy \le 2\ve
$$
uniformly on $\phi\in \mathcal K$ for all $\tau\ge \tau_0$, in particular for $\tau=h$.
 This makes the set $S_h(\mathcal K)$ precompact in $L^1(\ren)$.
\end{proof}

We resume the proof of existence. It now follows from the Schauder Fixed Point Theorem,  see \cite{Evansbk}, Section 9,
that there exists at least  a   fixed point $\phi_h \in \mathcal{K}$, \textit{i.\,e.,} $ S_h(\phi_h) = \phi_h$. The fixed point is in $\mathcal{K}$, so it is not trivial because  its mass is $1$.
Iterating the equality we get periodicity for the orbit $V _h(y, \tau)$  starting at $\tau = 0$
from $V_h(y,0)=\phi_{h}(y)$:
$$
V_h(y,\tau+ kh) =  V_h(y,\tau )\quad  \forall \tau > 0,
$$
This is valid for all integers $k\geq1$.   It is not a trivial orbit, $ V _h\not\equiv 0$.

\medskip

   $\bullet$\textit{Step 2. Periodicity of solutions.}    The next result examines the role of periodic solutions.

\begin{lemma}\label{lemma.periodicmeans}
Any periodic solution  of our renormalized problem, like $ V_h$, must be stationary in time. We will write $V_h(y,\tau )=F(y)$.
\end{lemma}

\begin{proof} The proof follows the lines of the uniqueness proof of previous subsection. Thus, if $ V _1$ is periodic solution that is not stationary, then $ V _2(y,\tau)= V _1(y,\tau+ c)$ must be different from $ V _1$ for some $c>0$, and both have the same mass. With notations as above we consider the functional { defined in \eqref{J} evaluated at $V_1,V_2$:}
 \begin{equation*}
 J[ V_1, V_2](\tau)=\int_{\mathbb{R}^N} ( V_1(x,\tau)- V_2(x,\tau))_+\,dx.
 \end{equation*}
 {As recalled before, b}y the accretivity of the operator  this is a Lyapunov functional, i.e., it is nonnegative and nonincreasing in time.
 By the periodicity of $ V _1$ and $ V _2$, this functional must be periodic in time. Combining those properties we conclude that it is constant, say $C\ge 0$, and we have to decide whether $C$ is a positive constant or zero.  In case $C=0$ we conclude that $V_1=V_2$ and we are done.

To eliminate the other option, $C>0$, we may argue as in the uniqueness result. We will prove that for two different solutions with the same mass this functional must be strictly decreasing in time. The main point is that such different solutions with the same mass must intersect. We argue as in Subsection \ref{sec.uniq}, where the difference is due to the fact that our solutions are \emph{not} in the self-similar form : we define at a certain time, say $\tau=0$, the maximum of the two profiles $ V ^*(0)=\max\{ V _1(0), V _2(0)\}$, and the minimum $ V _*(0)=\min\{ V _1(0), V _2(0)\}$. Let $ V ^*$ and $ V _*$ the corresponding solutions for $\tau>0$. We have for every such $\tau>0$
  \begin{equation}\label{eq.order2}
 V _*(y,\tau)\le  V _1(y,\tau),  V _2(y,\tau)  \le  V ^*(y,\tau).
\end{equation}
On the other hand, it easy to see by the definitions of $ V ^*(0)$, $ V _*(0)$ that
$$
\int_{\mathbb{R}^N}  V ^*(y,0)\,dy=M+ J[ V _1, V _2](0), \quad  \int_{{\mathbb{R}^N}}  V _*(y,0)\,dy=M- J[ V _1, V _2](0).
$$
Since $V ^*(y,0)$ and $ V _*(y,0)$ are ordered, this difference of mass is conserved in time: for $\tau>0$
\begin{equation}
\int_{{\mathbb{R}^N}} ( V ^*(y,\tau)- V _*(y,\tau))\,dy= 2J[ V _1, V _2](0).\label{eqJ2}
\end{equation}
Now, since $ V _{1},\, V _{2}$ have the same mass, \eqref{eq.order2} and \eqref{eqJ2} imply that for $\tau>0$
$$
\int_{\mathbb{R}^N} ( V _2(y,\tau)- V _1(y,\tau))_+\,dy=\int_{{\mathbb{R}^N}} ( V _1(y,\tau)- V _2(y,\tau))_+\,dy\le J[ V _1, V _2](0),
$$
but the constancy of  $J[ V _1, V _2]$ forces to have an equality, occurring only if the solution $ V ^*(y,\tau)$ equals the maximum of the two solutions $ V _1$ and $ V _2$, and the solution $ V _*(x,\tau)$  equals the minimum of the two solutions.

We argue then as in Subsection 6.1, in order to show that the constant defining $J$ is actually zero, therefore $V_{1}=V_{2}$. We need only to recall that $ V _1$ and $ V _2$ are SSNI. \end{proof}

\medskip

  $\bullet$\textit{Step 3. Conclusion.}   Thus, we take as profile the just defined $F(y)=V_h(y,0))$, where $V_h$ is the fixed point found above. Going back to the original variables, it means that the
corresponding function
\begin{equation}\label{u cap}
U(x,t)= t^{-\alpha} F(x_1 t  ^{-\sigma_1\alpha},...,x_N  t  ^{-\sigma_N\alpha})
\end{equation}
is a self-similar solution of equation \eqref{APM}  by construction. Indeed, it is defined as a self-similar function in \eqref{u cap} and  the profile $F$ verifies the stationary equation \eqref{StatEq}, see Lemma \ref{selfsimarefundsolut}. Actually $U$ has mass 1, but we can get a self-similar solution of any fixed mass by using the scaling \eqref{scal.tr}.


From this moment on we will denote the fundamental solution with mass $M$ by the label $U_M(x,t)$ and its profile, given by \eqref{sss}, by $F_M(y)$. The subscript $M$ will be omitted at times when explicit mention is not needed.

\subsection{Property of monotonicity with respect to the mass}\label{ssec.monotM}

 We conclude this section with some properties of the self-similar fundamental solution $U_M$ of mass $M$ and on its profile $F_M$, both built in the previous subsection. First we prove the monotonicity property with respect to the total mass. This property will be needed below.

\begin{proposition}\label{prop.mon.M} The profile $F_M$ depends monotonically on the mass $M$.
\end{proposition}

\noindent {\sl Proof.} Let us suppose $M_2>M_1>0$. We will prove that  $F_{M_2}(y)\ge F_{M_1}(y)$ for all $y$.

i) By uniqueness of the profile of every mass (see Section \ref{sec.uniq}) and \eqref{Tk}, we have
$$
F_{M_2}(y)=k F_{M_1}(k^{(1-m_1)/2}y_1,\cdots,k^{(1-m_N)/2}y_N)
$$
where $k>1$ is such that $M_2=M_1k^{\frac{N}{2}(\bar m-m_c)}$. Then $F_{M_2}(0)>F_{M_1}(0)$. Moreover $F_{M_2}(0),F_{M_1}(0)>0$ and by continuity $F_{M_1},F_{M_2}$ are positive in a  neighbourhood of zero $I(0)$.

ii) We are as in the conclusion of the proof of uniqueness (Subsection \ref{sec.uniq}). Let us consider $F^*(y)=\max\{F_{M_1},F_{M_2}\}$. It is a solution of the equation and $F^*(y)=F_{M_2}(y)$ in $y=0$. We stress that in the open set $I(0)$ the solutions $F^*,F_{M_2}$ are positive, so that we can prove locally smoothness for them since the equation is not degenerate (see \cite[ Theorem 6.1, Chapter V]{LSU}). As a consequence, we can apply the strong maximum principle in the whole open set $\Omega_2$, where $F_{M_2}>0$. We conclude that $F^*(y)=F_{M_2}(y)$ in $\Omega_2$. If $\Omega_2$ is the whole of $\mathbb{R}^N$, we have arrived at the conclusion that $F_{M_2}(y)\ge F_{M_1}(y)$ for all $y$.

iii) We still have to consider the case where $\Omega_2$ is a proper subset of $\mathbb{R}^N$.
 We observe that by Lemma \ref{star shap}, the set $\Omega_2$ is  star-shaped sets from the origin. Then for every unit direction $\vec e$ there is a point $x=s_0\vec e$, $s_0>0$, that belongs to the boundary of $\Omega_2$ and is such that $F_{M_2}(s\vec e)>0$ if $s<s_0$ and $F_{M_2}(s\vec e)=0$ if $s\ge s_0$. We conclude from the previous analysis that $F_{M_1}(y)\le F_{M_2}(y)=0$ on $\partial \Omega_2$, which means by the property SSNI applied to $F_{M_1}$ that $F_{M_1}=0$ is zero outside of $\Omega_2$. The conclusion is that $F_{M_2}(y)\ge F_{M_1}(y)$
 everywhere. \qed

\subsection{Positivity of the self-similar fundamental solutions}\label{ssec.posit}

Now we prove the strict positivity of the self-similar fundamental solution. 

\begin{theorem}\label{TH2} The  self-similar fundamental solution $U_M(x,t)$ is strictly positive for every $x\in\ren$, $t>0$. Its  profile function $F_M(y)$ is a $C^\infty$  and positive function everywhere in $\ren$. Moreover, there are sharp lower estimates of the asymptotic behaviour when $|y|\to\infty$ in the form
\begin{equation}\label{lowerbh}
F_1(y)\ge  c\,\min_i\{|y_i|^{-2/(1-m_i)}\}.
\end{equation}
\end{theorem}

\noindent {\sl Proof.}
 (i)  We first recall that the mass changing transformation  \eqref{Tk} with $\nu_i=(1-m_i)/2$ maps solutions of the stationary equation $\eqref{StatEq}$ of mass $M$ into solutions of the same equation of mass $k^\beta M$ with $\beta=\frac{N}{2}(\bar m-m_c)$ for every $k>0$. In particular, if $F_1(y)\ge 0 $ is the profile of the self-similar solution with unit mass, then $F_M(y)= (\mathcal T_k F_1)(y)$ is the profile of the self-similar solution with mass $M=k^{\beta}$.

(ii) Now, by Proposition \ref{prop.mon.M} we know that the family $F_M(y)$ is monotone nondecreasing in $M$, hence with respect to $k$.  It follows that for every choice of initial point $y_0=(y^0_1,\dots,y^0_N)$ the function
$$
f(k)= k\,F_1 (y^0_1 k^{\nu_1},\cdots, y^0_N k^{\nu_N})
$$
is increasing as $k$ increases. By the quantitative positivity lemma, Lemma \ref{lem.pos},  we also know that the function $F_1(y)$ is positive in a small box $Q_{r_0}=[-r_0,r_0]^N$:
\begin{equation*}
F_1(y)\ge  \ve>0 \quad \mbox{ for all } \ y\in \ren, \quad |y_i|\le r_0\, \quad \forall i.
\end{equation*}
(iii) Pick now any point $Y=(Y_1,\cdots,Y_N)$ outside of $Q_{r_0}$ and find the parameter $k$ such that
$|k^{-\nu_i}Y_i|\le r_0,$ for all $i$, i.e.,
$$
k=\max_i \left(\{|Y_i|/r_0\}^{2/(1-m_i)}\right)\,.
$$
We stress that $k>1$, because $Y\not \in Q_{r_0}$.  Let us take $(\widetilde{Y}_1,\cdots, \widetilde{Y}_N)=(k^{-\nu_1}Y_1,\cdots, k^{-\nu_N}Y_N)$. Since we have  $\widetilde{Y}\in Q_{r_0}$,  we get
$$
k\,F_1(Y)=k\,F_1(k^{\nu_1}\widetilde{Y}_1,\cdots,k^{\nu_N}\widetilde{Y}_N)\ge \,F_1(\widetilde{Y})\ge \ve.
$$
Therefore, $F_1(Y)$ is also positive for every $Y\not\in Q_{r_0}$. Moreover, the quantitative estimate $F_1(Y)\ge \ve/k$ can be written in the form:
\begin{equation}\label{lowerb}
F_1(Y)\ge  c\,\min_i\{|Y_i|^{-2/(1-m_i)}\},
\end{equation}
with $c=\ve  \min_i  r_0^{2/(1-m_i)}= c(N,m_i)$.

    \medskip

(iii) Using transformation \eqref{Tk} with $\nu_i=(1-m_i)/2$ we generalize  this lower estimate  to $F_M(y)$ for all $M>0$. Note that the lower bound \eqref{lowerb} is not affected by the change of mass, a curious propagation property that was already known in isotropic fast diffusion (where the self-similar solutions are explicit). What we find here is the correct form that is compatible with anisotropy.

\medskip

(iv) The global positivity of $U_M(x,t)$ is immediate. We have
$$
U_M(x,t)\ge c\, t^{-\alpha}\,\min_i\{|x_i|^{-2/(1-m_i)}t^{2a_i/(1-m_i)}\},
$$
with $\alpha$ and $a_i$ as in \eqref{alfa} and \eqref{ai}.
 \qed

\subsection{Regularity of the self-similar solution}\label{ssec.reg}

 The profile $F_{{M}}$ solves the quasilinear elliptic equation of the form \eqref{StatEq} which is singular in principle due to the nonlinearities $F_{M}^{m_i}$ with $0<m_i<1$. However, the regularity theory developed in great detail for nonsingular and non-degenerate elliptic equations, \cite{DiBenbk, LU}, applies to this case since it has a local form and we known that $F_{M}$ is  positive with quantitative positive upper and lower bounds in any neighbourhood of any point $y\in\ren$. Using well-known bootstrap arguments
 we may conclude

 \begin{theorem}\label{TReg} The  self-similar profile $F_M(y)$ is a $C^\infty$ function of $y$ in the whole space.
\end{theorem}

\section{Positivity for general nonnegative solutions}\label{pos.gen}

 We have just proved the strict positivity of self-similar solutions and given a positive lower bound for the rescaled function $v(y,\tau)$ introduced in \eqref{NewVariables} taking $t_0=1$ and then $\tau_0=\log t_0=0$.

\begin{theorem}[Infinite propagation of positivity]\label{thm.genpos}
Any  solution with nonnegative data and positive mass is continuous and positive everywhere in $\ren\times (0,\infty)$. More precisely, in terms of the $v$ variable, for every $R>0$  and $ \tau_0>0$ there exists a constant $C_2=C_2(v_0, R,  \tau_0 )>0$ such that $v(y,\tau)\geq C_2$ for $y\in B_R(0)$ and any $\tau>\tau_0$.
\end{theorem}

\noindent {\sl Proof.}  By solution we understand the solutions constructed in Theorem \ref{EUWES}.  We split it the proof into two cases.

\medskip

\indent {\bf I. Special data.}  {\sl The positivity result is true if the initial function $u_0$ is continuous, SSNI and compactly supported in a neighbourhood of the origin}.

 (1) We take those special data and get a lower estimate in small balls. For  SSNI see definition in Section \ref{ALEKSAND}. Arguing in terms of the rescaled variables, the assumptions  guarantee that the rescaled solution $v(y,\tau)=(t+1)^{\alpha}u(x,t)$ has initial data $u_{0}=v_{0}\leq \mathcal{T}_k\overline{F}$, where $\mathcal{T}_k\overline{F}$ is a suitable barrier (supersolution) to \eqref{StatEq} in a certain outer domain  ${{\mathbb{R}^N\setminus\mathcal{T}_k\widetilde{B}_r(0)}}$ obtained by {{cutting and }}rescaling $\overline{F}$ {{defined by \eqref{outer.barr bis}}} (see Section \ref{sec.upp} and formula \eqref{Tk}). Then, we can apply the qualitative lower estimate, Lemma \ref{lem.pos}, and conclude that $v(y,\tau)\ge \zeta(y)\ge c_1>0$ in a neighbourhood of $y=0$ for any time. We stress that $c_1$ depends on the $L^1$  norm of $u_0$ and on the radii $R$   and $r_0$ that are defined and used in Lemma \ref{lem.pos}.

(2)  In the sequel we must also work with $u$, since it satisfies a translation invariant equation, and this property is useful. From the lower bound for $v$ a corresponding lower bound formula holds for $u(x,t)$ in any time interval $0<t<t_1$, but this bound cannot be uniform in time.
Indeed, the lower bound for $u$, let us call it $c_1(u)$, depends on the final time $t_1$. We stress that we can make $t_1$ as large as we want  by taking $c_1(u)$ small enough. As a compensation, the decaying lower estimate applies to $u(x,t)$ in $x$-balls  that expand coordinate-wise like powers of time. This is a consequence of the rescaling in space.

Moreover, note that the argument works if we displace the origin and assume that $u_0$ is SSNI around some $x_0\ne 0$. In order to get a convenient definition of the rescaled variables $v(y,\tau)$ we must use the shifted space transformation $y_i=(x_i-x_{0i})t^{-\alpha \sigma_i}$. The previous argument shows that this $v$ will be uniformly positive in a given small neighbourhood of $0$ for all times. We conclude from this step that under the present assumptions $u(x,t)$ will be positive  forever in time in a suitable $x$-ball centered at $x_0$ that expands power-like with time, though the upper bound for $u$ decays like a power of time.

(3)  We now get an outer estimate under the previous assumptions ($u_0$ is continuous, compactly supported and SSNI).  By using the positivity of the  self-similar fundamental solution (see Theorem \ref{TH2}  ) we will prove that $u(x,t)$ is also positive in an outer cylinder $Q_o =D\times (0,t_2)$, where $D$ is the complement of  the ball of small radius $R_1$.  The idea is to find a small self-similar solution $U_{A}(x,t)$ and prove that
\begin{equation}
u(x,t)\geq U_{A}(x,t)>0
\end{equation}
for $(x,t)\in Q_o$. Since both are solutions of the equation in $Q_o$ we will check the comparison at the initial time and at the lateral boundary , and then we may apply the comparison principle Lemma \ref{comparisboundedunbound}  to conclude  that $u(x,t)\geq U_{A}(x,t)>0$ in the whole outer cylinder $Q_o$.

 The initial comparison is trivial since the fundamental solution vanishes for $t=0$ if $x\ne 0$ and $u_0\ge 0$ in $D$.
Comparison at the lateral boundary  is more delicate.  We use the fact that
 $$v(y,\tau)\geq c_{1}>0
 $$
  in the ball  $y\in B_{2R_1}(0)$. Therefore,
$$
u(x,t)\geq {c_1}\,{(t+1)^{-\alpha}}
$$
 for $x\in B_{2R_1}(0),$ $t>0$.

 We now use the scaling transformation \eqref{Tk} of the profile $F_1$ of the self-similar fundamental solution $U$ of mass $M=1$ and write, for every parameter $A>0$,
\[
F_{{A}}(y)=A \,{F_1}\left(A^{\nu_i}y_i\right)
\]
and let us consider the corresponding self-similar fundamental solution $U_{A}$. Since
\[
\lim_{t\rightarrow 0}U_{A}(x,t)=0 \quad \text{for }x\neq0 ,
\]
 we can choose $t<\bar{t}$, where $\bar{t}<t_{2}$, such that for any $x$ with $|x|=R_1$ one has
\[
U_{A}(x,t)<\frac{c_{1}}{(t_{2}+1)^{\alpha}}, \qquad t<\bar{t}.
\]
On the other hand, if $t\in [\bar{t},t_{2}]$ we have
\[
U_{A}(x,t)\leq \frac{{{1}}}{t^{\alpha}}{{F_A}}(0)\leq\frac{A}{\bar{t}^{\alpha}}{{F_1}}\left(0\right).
\]
At this point we can take $A$ so small (depending on $t_2$) such that
\[
U_{A}(x,t)\leq \frac{c_{1}}{(t_{2}+1)^{\alpha}}.
\]
We  have chosen so far a positive self-similar solution $U_{A}$ such that for all $t\in (0,t_2)$ we find
\[
u(x,t)\geq U_{A}(x,t)
\]
for $|x|=R_1$.  Given the initial and boundary comparisons between $u$ and $U_A$, we may apply the comparison principle  on exterior domains to conclude  that $u(x,t)\geq U_{A}(x,t)$ in the whole outer cylinder $D\times (0,t_2)$,  hence the positivity of $u$ in that set. The length of $t_2$ depends of the boundary conditions of the functions we compare. But using a solution with larger constant $A$ we can take $t_2$ as large as we like, with a worse lower estimate valid up to $t=t_2$. This concludes the proof of positivity in the case of solutions with special data.

(4) As for the quantitative aspect, we can combine the bounds in a small domain and the corresponding outer domain into an estimate of the form
$$
u(x,t)\ge U_M(x, t+t_1)
$$
valid for some $t\ge t_1$, $x\in\ren$  and $M>0$ small enough. Translating into the $v$ variable the quantitative estimate follows. \qed

\medskip

\noindent {\bf II. General data.} {\sl { Let us consider a general integrable initial datum $u_0\ge 0$ with positive mass.}}

(1) According of regularity of weak solutions (see Subsection 2.1)   the nonnegative solutions are continuous in space and time for all positive times. Since the mass of the solution is preserved in time and the solution is continuous, then given any $t_0>0$ we may pick some $x_0\in\ren$ such that $u(x,t)\ge c_1$ for some constant $c_1>0$ in a neighborhood of $(x_0,t_0)$.  We can choose a small function $w(x)$ that is SSNI around $x_0$, compactly supported and such that $w(x)\le u(x,t)$ for $\bar t$ close to $t_0$. By the proof of the special case we have that for $\ve>0$ small enough the solution $u_1(x,t)$ starting at $\bar t=t_0-\ve$ with initial value $u_1(x,t_0-\ve)=w(x)$ is positive, and by comparison $u(x,t)\ge u_1(x,t)>0$  for all $x$ and for $t_0-\ve<t<t_0+t_2-\ve$. After checking that $t_2$ does not depend on $\ve$ we conclude that $u(x,t_0)>0$  for all $x$. We have obtained the infinite propagation of positivity of $u$ because $t_0$ is any positive time.

(2) A careful analysis of the argument shows that given any finite radius $R$  and $\tau_0$ small enough,  we can find a uniform lower bound for $v(y,\tau) $ valid for $y\in B_R(0)$ and any $\tau>\tau_0$.

\medskip

\begin{remark}
 We cannot obtain a uniform lower bound from below in the whole space since the solutions are supposed to decay as $|y|\to\infty$, like in the fundamental solution, see Theorem \ref{fundamental solution}.
\end{remark}

\begin{remark}\label{stima uniforme v}
{As a consequence of Theorem \ref{thm.genpos} and \eqref{Linfty-L1} we have that $v(y,\tau)$ has a lower and upper bound in a every space ball or in every space cube $Q(R)$ of side $R>0$ for every $\tau$ big enough ($\tau>\tau_1$). This means that in this interior parabolic cylinder $Q(R)\times (\tau_1,\infty)$, the solution is bounded above and below away from zero,
so it is the solution of a quasilinear non-degenerate equation, so that according to the  standard regularity theory for uniform elliptic and parabolic equations, see \cite{LSU}, the solutions are smooth inside the cylinder with uniformly local estimates on the derivatives.}
\end{remark}


\section{Asymptotic behaviour}\label{sec.asymp}

Once the unique  self-similar fundamental solution  $U_M$ of the form \eqref{sss} is determined for any mass $M>0$ , the experience in Nonlinear Diffusion dictates that it is natural to expect it to be the candidate to be attractor for a large class of  solutions to the Cauchy problem for equation \eqref{APM}. We have the stated the result in Theorem \ref{thmasympto} and we will prove it here.  We recall that $U_M$ is a self-similar fundamental solution stated in Theorem \ref{fundamental solution}.

\medskip

\noindent {\sl Proof  of Theorem \ref{thmasympto}. }
The proof    follows a method that is similar to the so-called ``four-step method'',
a general plan to prove asymptotic convergence devised by Kamin-V\'azquez \cite{KV88} and applied to the isotropic case, see also \cite[Theorem 18.1]{Vlibro}. Here, a number of variations  are needed to take care of the peculiarities of the  anisotropy.  For example, here we have uniqueness only for self-similar fundamental solutions with fixed mass $M$, while in the isotropic case we know the uniqueness of the fundamental solution.

The main effort will be concentrated on proving the asymptotic convergence  for $p=1$:
\begin{equation}\label{L1conv.1}
\lim_{t\rightarrow\infty}\|u(t)-U_{M}(t)\|_{L^1(\ren)}=0.
\end{equation}

\medskip

\noindent (1) We introduce the family of \emph{rescaled} solutions given by the $u_{k}$'s in \eqref{uk}, namely
\[
u_{\lambda}(x,t)=\mathcal{R}_\lambda (x,t)=\lambda^{\alpha}u(\lambda^{\sigma_1\alpha}x_1,..., \lambda^{\sigma_N\alpha}x_N,\lambda t).
\]
We observe that the mass conservation and the $L^{1}$-$L^{\infty}$ smoothing effect \eqref{Linfty-L1} allow to find by interpolation the uniform boundedness of the norms $\|u_{\lambda}(\cdot,t)\|_{p}$ for all $p\in[1,\infty]$ and $t>0$. Moreover, using \eqref{Energywholespace} and the algebraic identity \eqref{ab} we find, for all $t>t_{0}>0$,\
\begin{align*}
&\int_{t_{0}}^{t}\int_{\R^{N}} \left|\frac{\partial u_{\lambda}^{m_{i}}}{\partial x_{i}}\right|^{2}dx\,d\tau=\lambda^{\alpha(2m_{i}+2\sigma_{i}-1)-1}
\int_{\lambda t_{0}}^{\lambda t}\int_{\R^{N}} \left|\frac{\partial u^{m_{i}}}{\partial x_{i}}\right|^{2}dx\,d\tau \\
&\leq \lambda^{\alpha(2m_{i}+2\sigma_{i}-1)-1}\int_{\R^{N}}\left|u(x,\lambda t_{0})\right|^{m_{i}+1}dx\le  M  \lambda^{\alpha m_i} \|u(\cdot,\lambda t_0)\|_{L^\infty}^{m_i},
\end{align*}
and using the smoothing effect \eqref{Linfty-L1} we get
\begin{equation}\label{uniformspatdervlambda}
\int_{t_{0}}^{t}\int_{\R^{N}} \left|\frac{\partial u_{\lambda}^{m_{i}}}{\partial x_{i}}\right|^{2}dx\,d\tau\le C M^{ 1+2m_i\frac{\alpha}{N}}t_0^{-\alpha m_i},
\end{equation}
an estimate that is independent of $\lambda$. Thus,   for all $i$ the derivatives $\partial_{x_{i}}(u^{m_{i}}_{\lambda})$ are equi-bounded in $L^{2}_{x,t}$ locally in time.  Moreover, the $L^{1}$-$L^{\infty}$ smoothing effect \eqref{Linfty-L1} implies that $\left\{u_{\lambda}\right\}$ is equi-bounded in $L^{\infty}$  for $t\geq\varepsilon$.

(2) We continue with an equi-continuity argument. 
We may use Remark \ref{stima uniforme v} to derive the uniform local equi-continuity of the function $v(y,\tau)$ defined by  \eqref{NewVariables}.
 Then, along subsequences $\tau_k\to\infty$, we get (let us use $t_0=1$)
\[
v(y,\tau_k)\rightarrow \widehat F(y)\quad \textit{as } \tau_k\rightarrow+\infty,
\]
uniformly on compact sets on $\R^{N}$. We can easily check that $u_{\lambda}(x,t)$ is related to the $v(y,\tau)$ by the formula
$$
u_{\lambda}(x,t)=
\lambda^{\alpha}(\lambda t+1)^{-\alpha}v(\lambda^{\sigma_1\alpha}(\lambda t+1)^{-\alpha\sigma_1}x_1,..., \log(\lambda t+1))
$$
We now let $\lambda$ be large enough and consider a fixed finite range $t_1<t<t_2$. The relation then simplifies into
$$
u_{\lambda}(x,t)\approx
 t^{-\alpha}v(t^{-\alpha\sigma_1}x_1,...,t^{-\alpha\sigma_N}x_N, \log t +\log\lambda)
$$
\noindent as $\lambda\rightarrow+\infty$, which immediately shows that $\lambda $ only appears significantly as a time shift in $\tau$, the rest of the expression is almost the same.

 We also need to consider the dependence in time to conclude that the family $\left\{u_{\lambda}\right\}$ is relatively compact in $L^{1}_{loc}(\R^{N}\times(0,\infty))$.  We may use the argument of Lemma \ref{lemma.precomp} to that purpose.
We conclude that along a subsequence  $\lambda_k\rightarrow+\infty$

there is a continuous function $\widetilde{U}(x,t)$ such that
\begin{equation}\label{lambdak}
u_{\lambda_k}(x,t)\rightarrow\widetilde{U}(x,t)\quad\text{ as }\lambda_k\rightarrow+\infty
\end{equation}
uniformly in each compact subset of $\R^{N}\times(0,+\infty)$.  The positivity theorem \ref{thm.genpos} implies that $\widetilde{U}$ is continuous.

\medskip

\noindent (3) We prove that $\widetilde{U}$ is a solution to \eqref{APM}. In order to pass to the weak limit in the weak formulation for the $u_{\lambda}$'s, we use the locally uniform convergence and uniform-in-time energy estimates of the spatial derivatives  $\partial_{x_{i}}u^{m_{i}}_{\lambda}$ for all $i=1,\cdots,N$, obtained in \eqref{uniformspatdervlambda}. Hence, adapting the proof of \cite[Lemma 18.3]{Vlibro} yields that $\widetilde{U}$ solves \eqref{APM} for all $t>0$. It must have a certain mass $M_1$ at each time $t>0$.

\noindent (3b) Now we show that the mass of $\widetilde{U}$ is just $M$. Arguing as in \cite[Theorem 18.1]{Vlibro}, first we further assume that $u_{0}$ is bounded and compactly supported in a ball $B_{R}(0) $ with mass $M$. Let us take $L_1>\sup u_0$ and a larger mass $M^{\prime}>M$ such that the upper barrier $ G_{k} (y)$ defined in \eqref{G_k} is such that $u_0(y)\leq  G_{k} (y)$. We recall that $ G_{k} (y)\in L^1(\mathbb{R}^N)$. By Theorem \ref{thm.barr} and change of variables (with $t_0=1$) it follows that
$$u(x,t)\leq (t+1)^{-\alpha}  G_{k} (x_1(t+1)^{-\alpha \sigma_1},\cdots,x_N(t+1)^{-\alpha \sigma_N})$$
for all $t>0$ and for all $x\in \mathbb{R}^N$. Then
\begin{equation}\label{STIMA 1}
u_\lambda(x,t)\leq \lambda^{\alpha} (\lambda t+1)^{-\alpha} G_{k} (\lambda^{-\alpha \sigma_1}x_1(\lambda t+1)^{-\alpha \sigma_1},\cdots,\lambda^{\alpha\sigma_N}x_N(\lambda t+1)^{-\alpha \sigma_N})
\end{equation}
for all $t>0$ and for all $x\in \mathbb{R}^N$. We observe that
\begin{equation}\label{STIMA 2}
\begin{split}
\lim_{\lambda\rightarrow +\infty}& \lambda^{\alpha} (\lambda t+1)^{-\alpha} G_{k} (\lambda^{\alpha \sigma_1}x_1(\lambda t+1)^{-\alpha \sigma_1},\cdots,\lambda^{\alpha\sigma_N}x_N(\lambda t+1)^{-\alpha \sigma_N})
\\
&=t^{-\alpha} G_{k} (t^{-\alpha \sigma_1}x_1,\cdots,t^{-\alpha \sigma_N}x_N)
\end{split}
\end{equation}
and the mass is preserved.
The previous facts, the convergence $u_{\lambda}\rightarrow  \widetilde{U} $ a.e. in $\mathbb{R}^N$ and \eqref{STIMA 1} allow to apply Lebesgue dominated convergence Theorem, obtaining
\[
u_{\lambda}(t)\rightarrow \widetilde{U}(t)\quad \text{in}\, L^{1}(\R^{N}),
\]
which means in particular that the mass of $\widetilde{U}$ is equal to $M$ at any positive time $t$. Thus, we have obtained that $\widetilde{U}$ is a fundamental solution with initial mass $M$. If we knew  this fundamental solution to be self-similar, then the uniqueness theorem would imply $\widetilde{U}(x,t)=U_M(x,t)$.

\medskip

(4)  We need another step to make sure that $\widetilde{U}(x,t)=U_M(x,t)$. In order to use ideas that are already in the paper we may use the Lyapunov functional
$$
J[u,U_M](t)=\int_{\ren} |u(x,t)-U_M(x,t)|\,dx.
$$
This is known to be nonnegative and non-increasing in time along solutions $u(x,t)$. Using the rescaling we get for all
$\lambda>1$
$$
J[u_\lambda,U_M](t)=\int_{\ren} |u_\lambda(x,t)-U_M(x,t)|\,dx= \int_{\ren} |u(y,\lambda t)-U_M(y,\lambda t)|\,dy=
J[u,U_{M}](\lambda t)
$$
(we use the scaling invariance $(U_M)_\lambda=U_M$), which proves that $J[u_\lambda,U_M](t)$ is non-increasing in $\lambda$ for fixed $t>0$. Therefore, we have the common limit
$$
\lim_{\lambda \to\infty}J[u_\lambda,U_M](t)=\lim_{t\to\infty} J[u,U_M](t) = C_\infty\ge0.
$$

\begin{lemma}\label{lemma.cinfty} We necessarily have $C_\infty=0.$
\end{lemma}
\begin{proof} We exclude the case $C_\infty>0$ as follows. Let $\lambda_k$ the sequence mentioned above that produces the limit $\widetilde U$ as in \eqref{lambdak}. Passing to the limit $\lambda_k\to \infty$ we get for every $t>0$
$$
J[\widetilde U,U_M](t)=\lim_{k\to \infty}J[u_{\lambda_k},U_M](t)= C_\infty>0.
$$
This is a peculiar situation where two solutions with the same mass have constant $L^1$ difference in time. We exclude the situation by considering $U^*$ the maximum of the two solutions and checking, as in point (iii) of the existence proof in Subsection \ref{existself} (see also the uniqueness proof in Subsection \ref{sec.uniq}), that $U^*$ must also be a solution of the equation. The argument follows by  observing that $U_M$ is positive everywhere, hence also $U^*$ is positive, and both are bounded for $t\ge 1$ by the smoothing effect. Moreover, $\widetilde U$ and $U^*$ are not the same for any $t>1$ since they differ in $L^1$ norm. Hence, they must intersect and there must be a point $x_0\in \ren$ such $U^*(x_0,1)=U_M(x_0,1)=\widetilde U(x_0,1)$. By continuity we know that $U^*(x,t)$ and $U_M(x,t)$ are bounded above and below away from zero in some neighbourhood $D$ of $(x_0,1)$, so they solve a quasilinear non-degenerate parabolic equation in divergence form in $D$. Since $U^*(x,t)\ge U_M(x,t)$, we can apply the strong maximum  principle \cite{LSU, PS2007} to conclude that $U^*(x_0,1)=U_M(x_0,1)$ is only possible if they also agree on a maximal connected domain, in particular for all $x\in\ren$ and $t=1$. This is a contradiction, hence  $C_\infty=0.$

\qed

(4b) Once we have $C_\infty=0$ we may join this and \eqref{lambdak} to get the conclusion that for any limit of a subsequence $\widetilde U=U_M$, and this is the $L^1$ convergence formula \eqref{L1conv} that we were aiming at. We recall that this was proved under the assumption that $u_0$ is bounded and has compact support.
The general case $u_{0}\in L^{1}(\R^{N}), u_0\ge 0$ follows as in \cite[Theorem 18.1]{Vlibro} by approximation of the initial data,  using the $L^1$-contractivity of the flow and the continuity of $U_M$ with
respect to $M$. Recall that given two self-similar fundamental solutions $U_{M_1},U_{M_2}$ with two different masses $M_1$ and $M_2$ we get $$\|U_{M_1}-U_{M_2}\|_{1}=\|F_{M_1}-F_{M_2}\|_{1},$$
where $F_{M_1},F_{M_2}$ are their profiles respectively. Proposition \ref{prop.mon.M} guarantees that for $M_1>M_2$ we have $F_{M_1}\ge F_{M_2}$ so that
$$
\|F_{M_1}-F_{M_2}\|_{1}=\int_{\ren} (F_{M_1}(x)-F_{M_2}(x))\,dx= M_1-M_2.
$$

\noindent {\bf $L^p$-convergence for $p>1$}. This is an easy consequence of the convergence in $L^1$ and uniform boundedness in $L^\infty$ by observing that
\begin{align*}
\|u(\cdot,\lambda)-U_{M}(\cdot,\lambda)\|^p_{p}& \leq \|u(\cdot,\lambda)-U_{M}(\cdot,\lambda)\|_1 \|u(\cdot,\lambda)-U_{M}(\cdot,\lambda)\|^{p-1}_\infty.
\end{align*}
 The first factor is estimated by the $L^1$ convergence \eqref{L1conv} as $o(1)$ when $t\to\infty$, while the terms $\|u(\cdot,\lambda)\|_\infty$ and $ \|U_{M}(\cdot,\lambda)\|_\infty$ are estimated as a constant times $t^{-\alpha}$ by the smoothing effect of Theorem \ref{L1LI}. In this way we get \eqref{Lpconv}. Note that in rescaled variables it reads
$$
\|v(\cdot,\lambda)-F_{M}(\cdot)\|_{p}\to 0 \quad \mbox{as} \ \lambda\to\infty,
$$
 recalling that $U_M$ is given by \eqref{sss} in terms of the self-similar profile $F_M$.
\end{proof}

\begin{remark}
\begin{itemize}
\item[i)] We stress that no    B\'enilan-Crandall estimate for the time derivative $\partial_{t}u$    \cite{BC81}   is available (contrary to the isotropic case), therefore in the proof we need a novel argument to obtain relative compactness in $L^{1}_{loc}(\R^{N}\times(0,\infty))$.

\item [ii)]  Since $\cup_{L_1\geq0}\mathcal K(L_1)$ is dense in $L^1(\R^{N})$, we could start to prove Theorem 1.2 for initial data in $\mathcal K(L_1)$, where $\mathcal K(L_1)$ is defined in Step 1 of Proof of existence of a self-similar solution (Section 6.2). Then, we could conclude by density and by $L^1$-contraction.
    \end{itemize}
\end{remark}

Actually, we have a stronger asymptotic convergence result under extra conditions.

\begin{theorem}\label{thm8.2} If the initial datum $u_0$ is nonnegative, bounded and compactly supported, then we also have
\begin{equation}\label{con L infty}
\lim_{t\rightarrow\infty}t^{\alpha}\|u(t)-U_{M}(t)\|_{\infty}=0,
\end{equation}
where $\alpha$ is given by \eqref{alfa}.
\end{theorem}
\begin{proof} From the proof of Theorem  \ref{thmasympto} we already know that the solution $v(y,\tau)$ converges uniformly on the  compact  sets of $\mathbb{R}^N$  to $F_{ M  }(y)$ as $\tau\rightarrow+\infty$, thus the only thing to check is the control of the tails. We can use  the explicit upper barriers of  Section \ref{sec.upp} or a large rescaling thereof to bound above our solution for all times and thus control the decay of our solution at spatial infinity for all times (see Theorem \ref{thm.barr}). If $Q(R)$ is the space cube of side $R>0$, we deduce that $\forall \varepsilon>0$ there exist $R=R(\varepsilon)$  such that
$$
\|v(\cdot, \tau)\|_{L^\infty(\mathbb{R}^N\setminus Q(R))}\leq \varepsilon \, \text{ for all } \tau>0.
$$
Then we conclude that
\[
\|v(\cdot,\tau)-F_{ M  }(y)\|_{L^\infty(\R^{N})}\rightarrow0,
\]
which translates into \eqref{con L infty}.
\end{proof}
\vspace{-0.5cm}
\section{Numerical studies}\label{se.numer}

In this section we show the results of numerical computations with the evolution process that show the appearance of an elongated self-similar profile. We compute  in 2 dimensions for simplicity and plot the level lines to show the anisotropy.
\vspace{-0.5cm}
         \begin{figure}[H]
           \centering
               \includegraphics[scale=0.28]{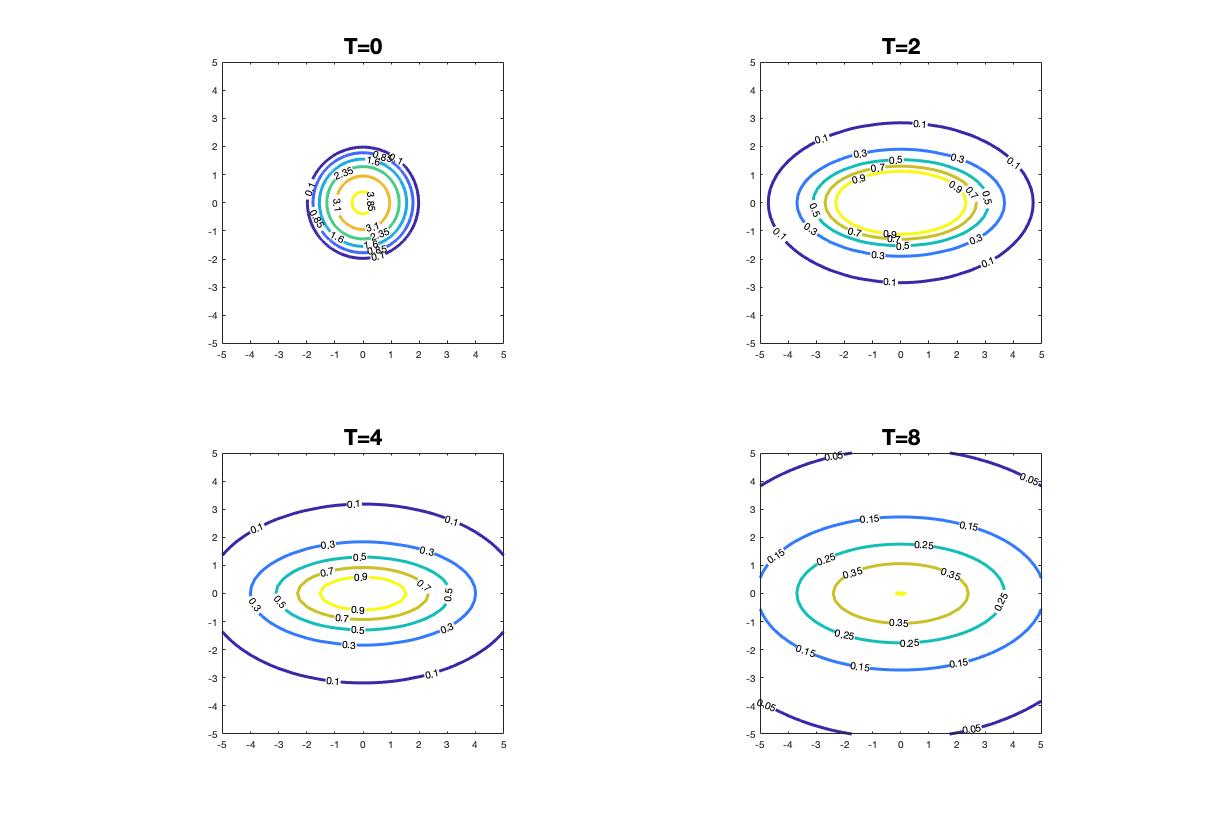}
\vspace{-1cm}
 \caption{Evolution from radial data to an anisotropic self-similarity}
  \label{fig3}
          \end{figure}
\vspace{-1.0cm}
          \begin{figure}[H]
           \centering
              \includegraphics[scale=0.28]{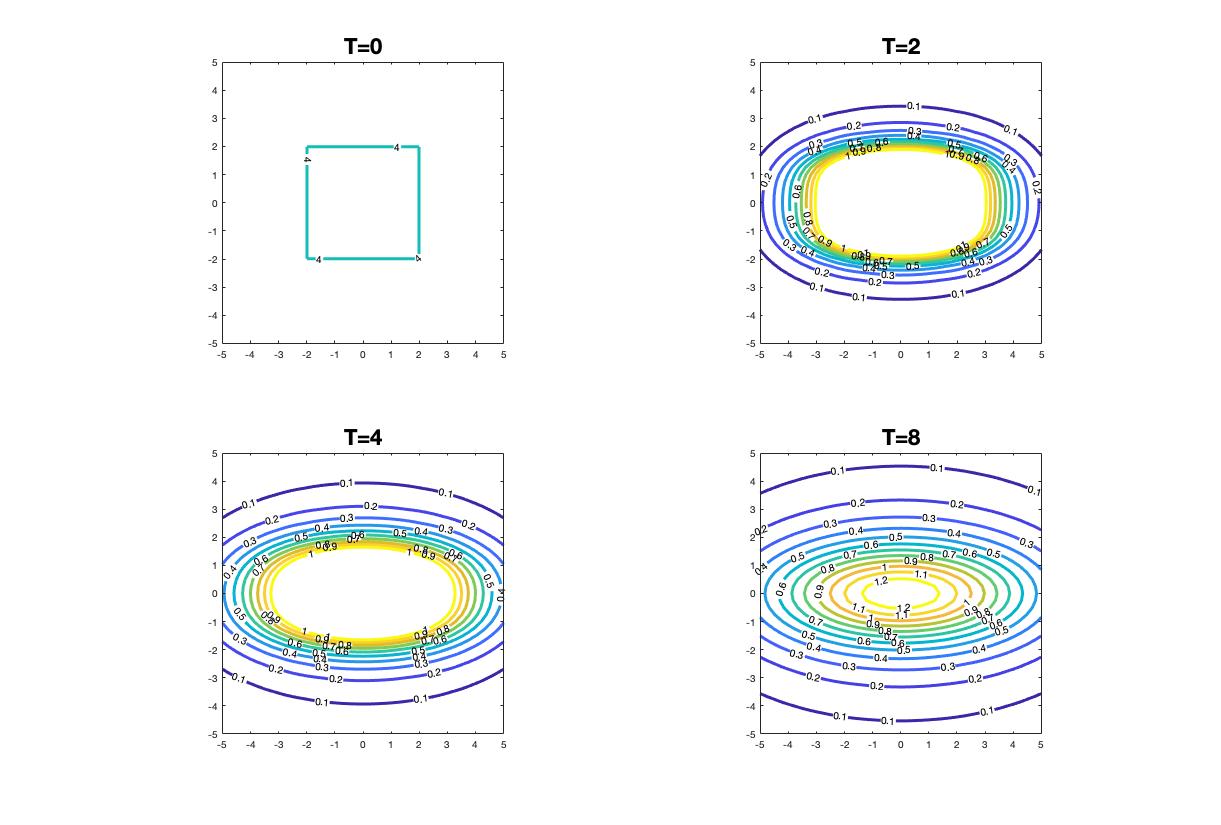}
              \caption{Evolution from a square configuration}
              \label{fig2}
          \end{figure}

\section{Fast diffusion combined with partial linear diffusion}\label{sec.meq1}

 This section contains a number of remarks when some of the $m_i$ equals one. If we revise the general theory: existence, uniqueness, continuity, smoothing effects and Aleksandrov principle, we see they work fine when one (or several exponents) are 1. Indeed it is possible to build an upper barrier (see Proposition \eqref{P1}) and the self-similarity works in the same way as before. Only the lower bound cannot be the same.

Let us make some computations. In particular, $m_1=1$ implies that
$$
\sigma_1=\frac1{N}+\frac{\overline{m}-1}2, \quad a_1=\alpha\,\sigma_1=\frac12,
$$
which is the heat equation scaling. On the other hand, if we write $x=(x_1,x')$, integrate in the rest of the variables
$x'=(x_2,\cdots, x_N)$, and put
$$
w(x_1,t)=\int_{\re^{N-1}}u(x_1,x', t)\,dx_2\dots dx_N\,,
$$
it is easy to see that $w$ satisfies a 1D Heat Equation: $w_t=w_{x_1x_1}$. When we apply the previous argument to a fundamental solution
we will find the 1D fundamental solution
$$
W(x_1,t)=(4\pi t)^{-1/2}e^{-x_1^2/4t}.
$$
If we write this formula  in terms of the fundamental solution profile we get
$$
\int_{\re^{N-1}} F(y_1,y')\,dy_2\dots dy_N=(4\pi)^{-1/2}e^{-y_1^2/4}.
$$
This means that in this direction the fundamental solution decreases in average like a negative quadratic exponential and not like a power.

\medskip

As announced before, when some $m_i=1$ we need to build a different barrier.

\begin{proposition}\label{P1}
Let \ $ \alpha, a_i, \sigma_i$ be defined in \eqref{alfa}-\eqref{ai}. Take $\delta>0$ and $\theta_i\geq 2$ such that
\begin{equation}\label{delta bis}
\frac1{\sigma_i}<\delta\theta_i<\frac2{1-m_i}\,.
\end{equation}
Let $\Omega_r=\{y\in \mathbb{R}^N: \ \sum_{i=1}^N|y_i|^{\theta_i}\ge r\}$ be an anisotropic outer domain, where
$r>0$ is given by
\begin{equation}\label{r}
r:=\max_i\left(\frac{N\delta m_i(\delta m_i+1)\theta_i^2}{(\delta\min_i\{\sigma_i\theta_i\}-1)\alpha}\right)^{\frac{1}{2/\theta_i-\delta(1- m_i)}}.
\end{equation}
Then the function
\begin{equation}\label{outer.barr}
\overline{F}(y)=\left( \sum_{i=1}^N |y_i|^{\theta_i} \right)^{-\delta}
\end{equation}
is a supersolution to equation \eqref{StatEq} in the domain $\Omega_r$ and  $\overline{F}\in L^1(\Omega_r)$.
\end{proposition}

\noindent {\bf Remarks.}
1) We observe that under the assumptions $0<m_i\leq1$ for all $i$, (H2) and recalling \eqref{ai}, condition \eqref{33} fulfills, where $1/(1-m_i):=+\infty$ if $m_i=1$.

2) In the choice of exponents for the supersolution we can take $\theta_i\delta$ as close as we want to the dimensional exponent $2/(1-m_i)$.

 3) Completing $\overline{F}$ inside the inner domain $D_r=\ren\setminus \Omega_r$ by the constant $\max_{y\in \Omega_r}\overline{F}(y)$ we obtain the global function
$$
G=\min\,\{\overline{F}(y),\max_{y\in \Omega_r}\overline{F}(y)\} \in L^1(\ren)
$$
This is the type of function we will use, after a suitable rescaling, as a barrier in our comparison theorem (see Theorem \ref{thm.barr}).

 4) For another upper barrier construction  see \cite[Lemma 2.3]{SJ05}.
\qed

\

\noindent {\sl Proof of Proposition \ref{P1}}. If we put $X=\sum_{j=1}^N|y_j|^{\theta_j}$, then since $\theta_i\geq1$ we get
\begin{equation*}
\begin{split}
I:=&\sum_{i=1}^{N}\left[(\overline{F}^{m_i})_{y_iy_i}+a_i\left( y_i \overline{F}\right)_{y_i}
\right]
\\
&\leq
\sum_{i=1}^{N}\delta m_i(\delta m_i+1)\theta_i^2X^{-\delta m_i-2}X^{2-2/\theta_i}
+\alpha\left[1-\delta\min_i\{\sigma_i\theta_i\}\right]X^{-\delta},
\end{split}
\end{equation*}
where $1-\delta\min_i\{\sigma_i\theta_i\}<0$ by \eqref{delta bis}. In order to conclude that $I\leq0$ it is enough to show that
\begin{equation*}
\left[\delta m_i(\delta m_i+1)\theta_i^2X^{\delta-\delta m_i-2/\theta_i}
+\frac{\alpha}{N}\left(1-\delta\min_i\{\sigma_i\theta_i\}\right)\right]\leq0
\end{equation*}
for every $i=1,..,N$, where $\delta-\delta m_i-2/\theta_i<0$ by \eqref{delta bis}.
Then we have to require $X\geq r$ with $r$ given by \eqref{r}. This together with  Lemma \ref{Lemma Song} completes the proof.\qed

In this range we set as barrier a suitable rescaled  (according to formula
\eqref{Tk}) $\mathcal{T}_k\overline{F}(y)$ of $\overline{F}(y)$, the function given in \eqref{outer.barr} defined in the exterior domain $\Omega_r$, defined in Proposition \ref{P1} (see Fig. \eqref{fig:boat2}).
We stress that in \eqref{Tk} $\nu_i=(1-m_i)/2$ $\nu_i=0$ if $m_i=1$. In this way the inner hole $D=\ren\setminus \Omega_r$ changes into
$$
{\mathcal T}_k D=\{y=(y_1,\cdots,y_N)\in \mathbb{R}^N: \ \sum_i(k^{\nu_i}y_i)^{\theta_i}< r\}
$$
that we can make  as small as we want if $k$ is large. Note that this changes the mass ({or the $L^1$ norm)
\begin{equation}\label{Change mass bis}
\int_{\mathcal{T}_k\Omega_{r}} F_k(y)dy= k\int_{\mathcal{T}_k\Omega_{r}} F(y_i\,k^{\nu_i})\,dy=k^{\beta} \,\int_{\Omega_r} F(z)\,dz\,,
\end{equation}
where $\beta=1-\sum_i\nu_i=1-N(1-\overline {m})/2\in (0,1)$ and we denote the rescaled set of $\Omega_r$ by
\begin{equation}\label{Tk Omega bis}
\mathcal{T}_k\Omega_r=\left\{y \in\mathbb{R}^N:(k^{\gamma_1}y_1,\cdots,k^{\gamma_N}y_N )\in \Omega_r \right\}=\{y\in \mathbb{R}^N: \ \sum_{i=1}^N k^{\gamma_i}|y_i|^{\theta_i}\ge r\}.
\end{equation}
In order to have a global barrier, we will extend $\mathcal{T}_k\overline{F}(y)$ outside $\mathcal{T}_k\Omega_r$ by $\max\{\mathcal{T}_k\overline{F}(y): y\in \mathcal{T}_k\Omega_r\}$, \textit{i.e.,} the value it takes at the boundary of $\mathcal{T}_k\Omega_r$.
Then in Theorem \ref{thm.barr} the barrier \eqref{G_k}  has to be replaced by
\begin{equation}\label{G_k bis}
G_{k}(y)=\min\{
\mathcal{T}_{k}\overline{F}(y),
\max_{\mathcal{T}_{k}(\Omega_r)}\mathcal{T}_{k}\overline{F}(y)\}
\end{equation}
for every $k>0$, where $\mathcal{T}_k$ is defined in \eqref{Tk} and $\overline{F}$ is given in \eqref{outer.barr}.

\medskip

Finally we observe that Theorem \ref{contraction} holds if some $m_i=1$ as well but we will do some modifications to its proof.

\textit{Proof of Theorem \ref{contraction} if some $m_i=1$.} We need a variant of the argument of point (ii) presented in Subsection \ref{subsec_contraction}. We may assume that $m_i=1$  for $i=1,\dots,j_0$, $j_0<N$, and  $m_i<1$ for all $i> j_0$.

The idea is to fix the scaling of $\zeta_1$ as a factor $n^{-(1-m_i)}$ for the directions with $m_i<1$ (as before), and insert a factor $1/n^{\delta}$ for all $m_i=1$: to be more precise, we set
\[
\zeta_{n}(x)=\zeta_{1}(n^{-\delta}x_{1},...,n^{-\delta}x_{j_{0}},n^{-(1-m_{j_{0}+1})}x_{j_{0}+1},...,n^{-(1-m_{N})}x_{N}).
\]
Here, $\delta>0$ is small, as needed below. Repeating the above calculation, the terms with $m_i=1$ contribute to the  formula. Then we have
\begin{equation}\label{ODE1bis}
\frac{dX_n}{dt}\leq \sum_{i=1}^{j_{0}}\int_{\R^{N} } (u_{1}-u_{2})_{+}\,|\partial_{x_{i}x_{i}}\zeta_{n}(x)|dx +\sum_{i=j_{0}+1}^{N}\int_{\R^{N} } (u_{1}^{m_i}- u_{2}^{m_i })_{+}\,|\partial_{x_{i}x_{i}}\zeta_{n}(x)|dx.
\end{equation}
First we estimate the first term in the right-hand side of \eqref{ODE1bis} obtaining
\[
\sum_{i=1}^{j_{0}}\int_{\R^{N}} (u_{1}-u_{2})_{+}\,|\partial_{x_{i}x_{i}}\zeta_{n}(x)|dx\leq C\frac{1}{n^{2\delta}}
\int_{\R^{N}} (u_{1}-u_{2})_{+}dx,
\]
that goes to zero as $n\to\infty$.  Moreover, the second term in the right-hand side of \eqref{ODE1bis}, takes into account contribution of the terms with $m_i<1$. Then, arguing similarly as for the estimate \eqref{E4} we find

\begin{equation*}\label{E4}
\frac{dX_n}{dt}\le C\frac{1}{n^{2\delta}}
\int_{\R^{N}} (u_{1}-u_{2})_{+}dx+c\ve \,X_n(t)+\max_{i} C(\ve,m_{i})\,K^{\prime}_n,
\end{equation*}
where
$$K_n'=\sum_{i=j_0+1}^N (1-m_i)\,\int_{\R^{N}} \left(\zeta_{n}^{-m_i}\,|\partial_{x_{i}x_{i}}\zeta_{n}|\right)^{1/(1-m_i)}\,dx.$$
As in (iv) it is  easy to see that $K_n'= K'_1 n^{-\gamma'}$ with
$$
\gamma'= 2-\sum_{i=j_0+1}^N (1-m_i)-j_0 \delta.
$$
This quantity is still larger than zero if $\delta$ is small enough. We conclude as in Subsection \ref{subsec_contraction}.
\qed

There is no novelty in the Aleksandrov principle or the quantitative positivity lemma of Section  \ref{sec.qual} so that the proof of existence and uniqueness of a solution as in Theorem \ref{fundamental solution} can be done as in Section \ref{sec.ex.ssfs}. However, the quantitative study of positivity and the ensuing regularity need new work that we will address here because the paper is already long and the issue bears on different type of estimates.


\section{Comments, extensions and open problems}

 \noindent {$\bullet$} We can get explicit solutions with infinite mass and decay  estimates in terms of one-dimensional traveling waves. These solutions are explicit and decay in a chosen coordinate direction like
  $$
 u(x,t)\sim C(x_i-ct)^{-1/(1-m_i)}, \qquad x_i>ct.
 $$
Note  that the decay of the self-similar solution in $N$ variables along the $x_i$ axis is  $u(x,t)\sim |x_i|^{-2/(1-m_i)}$, while here the decay exponent is half as much. In any case the decay rate depends on the coordinate exponent. Note also that this fast diffusion wave is only defined in a moving domain $\{x_i>ct\}$ with a strong singularity on the lateral border of that domain. The solution is classical inside its domain.

 We can also compare solutions for different dimensions.

 \noindent {$\bullet$} We have the project of studying the existence of  self-similar fundamental  solutions for the slow case, also called Porous  Medium  case, where at least some of the $m_i$ are greater than 1. The main difference is the existence of compact support in some directions.


 \noindent {$\bullet$}  A very detailed analysis of the so-called  {anisotropic $p$-Laplacian evolution of fast diffusion type} was done subsequently by the authors in \cite{FVV21} following the main ideas of this paper. Variants, improvements and further details of the technique were described, as well as a comparison between results for both types of equation. For more information on quite general anisotropic $p$-Laplacian equations see for instance \cite{AS2015}.

  \noindent {$\bullet$} We have not studied what happens in the anisotropic FDE when $\overline{m}\le m_c$. The isotropic case is well-known by now and it is full of new phenomena and difficulties. See \cite{VazSmooth}   and \cite{DK07}.

  \noindent {$\bullet$} Can we accept negative powers $m_i<0$?  See as references in the isotropic case \cite{BBDGV, Vne}.

\noindent {$\bullet$}  Question: do we have explicit solutions in some cases? This happened in the isotropic case,  where ODE methods could be used, see \cite{BV2006, VazSmooth}.

\noindent {$\bullet$}  An interesting problem consists of posing our anisotropic equation in a bounded domain with suitable boundary conditions. We did not find an interesting relationship to our problem \eqref{APM}--\eqref{IC}.

\noindent {$\bullet$} Another open question: which anisotropic Gagliardo-Nirenberg-Sobolev inequalities would be natural ones associated to this anisotropic  fast diffusion equation flow?

\noindent {$\bullet$}    The isotropic analogue of asymptotic behaviour proposed in Theorem 1.2 is optimal (no rates in general are possible). Indeed,  sending the center of mass at infinity destroys any rate. Does the same phenomenon happen here ?

    Moreover, in the isotropic case the convergence in relative error or any Global Harnack principle holds true, as conjectured by Carrillo-V\'azquez \cite{CV03} and proven in \cite{BS}. Finally in literature appropriate tail
conditions allow to have sharp rates of convergence to equilibrium in rescaled variables (for example via
a nonlinear entropy method). These two questions could be investigated in the anisotropic case.

\medskip

\section{Appendix: proof of the smoothing effect}

The proof of Theorem \ref{L1LI} is given in Theorem 1.2 of Song-Jima \cite{SJ06}, but for reader's convenience we give more details. In order to do this we recall some anisotropic Sobolev inequalities.

Let us denote    $\widetilde{r}$   the harmonic mean of $   r_1,\cdots,r_N  \geq1$,\textit{ i.e.} $   \frac{1}{ \widetilde{r}}=\sum_{i=1}^N\frac{1}{r_i}.$

\begin{proposition}\emph{(}see \cite{DFZ}\emph{)}
Let $   \lambda_i >0$ and $1\leq \widetilde{p}<N$. Then for every nonnegative function $u\in C^\infty_0(\mathbb{R}^N)$ we have
\begin{equation}\label{Sobolev 0}
\left\|u^{\overline{{\lambda}}}\right\|_{L^{\widetilde{r}^*}}\leq C_S
\left\|\left(\prod_{i=1}^{N}\left|\partial_{x_i} u^{\lambda_i}\right|\right)^{1/N}\right\|_{L^{\widetilde{r}}},
\end{equation}
where $   \widetilde{r}^*=\frac{N\widetilde{r}}{N-\widetilde{r}} ,\overline{\lambda}=\frac{1}{N}\sum_{i=1}^N\lambda_i$ and $C_S$  is a positive constant depending on $N$ and $   \widetilde{r} $.
\end{proposition}

We stress that in \cite{DFZ} inequality \eqref{Sobolev 0} is proved for Lorentz norms. A proof using directly Sobolev norm can be obtained adapting Trois's proof \cite{Tr}. Now we use the following lemma to obtain the usual form of anisotropic inequalities that involves the product of the norms of the partial derivatives in $L^{   r_i  }$ with $   r_1,\cdots,r_N  \geq1$ (see \cite{Tr}).

\begin{lemma}\label{Lemma tetai}\emph{(}see \cite[page 43]{KPS}\emph{)} Let $X$ be a rearrangement invariant space and let $
0\leq \theta _{i}\leq 1$ for $i=1,...,M,$ such that $\sum_{i=1}^{M}\theta _{i}=1$, then
\begin{equation*}
\left\| \prod_{i=1}^{M}|f_{i}|^{\theta _{i}}\right\| _{X}\leq
\prod_{i=1}^{M}\Vert f_{i}\Vert _{X}^{\theta _{i}} \quad \forall f_i\in X.  \label{holder
X}
\end{equation*}
\end{lemma}

Indeed taking $\frac{1}{\widetilde{   r  }}=\frac{1}{N}\sum_{i=1}^N\frac{1}{   r_i  }$, $X=L^{\widetilde{   r  }}(\Omega)$, $\theta_i=  \frac{\widetilde{r}}{r_iN} $, Lemma \ref{Lemma tetai} and \eqref{Sobolev 0} yield  that
\begin{equation}\label{Sobolev prod}
\left\|u^{\underline{{\lambda}}}\right\|_{L^{\widetilde{   r  }^*}}\leq C_S \prod_{i=1}^N\left\|\partial_{x_i}u^{\lambda_i}\right\|^{1/N}_{L^{   r_i }}\quad \forall u\in C^\infty_0(\mathbb{R}^N).
\end{equation}
Finally using the well-known inequality between geometric and arithmetic means we get
\begin{equation}\label{Sobolev sum}
\left\|u^{\underline{{\lambda}}}\right\|_{L^{\widetilde{   r  }^*}}\leq \frac{C_S}{N} \sum_{i=1}^N\left\|\partial_ {x_i}u^{\lambda_i}\right\|^{1/N}_{L^{   r_i  }}\quad \forall u\in C^\infty_0(\mathbb{R}^N).
\end{equation}

\medskip

In the case $\widetilde{   r  }=N$ the following result holds.

\begin{proposition}
Let $\lambda_i\geq0$ (but not both identically zero) and $   r_1,\cdots,r_N  \geq1$ be such that $\widetilde{   r  }=N$. Then for every nonnegative functions $u\in C^\infty_0(\mathbb{R}^N)$ we have
\begin{equation}\label{Sobolev pN bis}
\left\|u^{\underline{{\lambda}}}\right\|_{L^{q}}\leq \frac{K_S}{N}\left[\left\|u^{\underline{{\lambda}}}\right\|_{L^N}+
\sum_{i=1}^{N}
\left\|\partial_{x_i} u^{\lambda_i}\right\|_{L^{   r_i  }}\right]
\end{equation}
for all $q\geq N$ and $K_S$  is a positive constant depending on $N$ and $   r_1,\cdots,r_N  $.
\end{proposition}

\noindent {\sl Proof.}
We can argue as in \cite[Corollary IX.11]{Libro Brezis} starting from \eqref{Sobolev 0} with $\widetilde{   r  }=1$. At the end we apply  Lemma \ref{Lemma tetai} and the well-known inequality between geometric and arithmetic means to conclude.
\qed

\bigskip

As a first step we obtain a bound of the $L^\infty$ norm in terms of the $L^p$ norm of the initial
datum for every $p>1$.

\begin{theorem} Let assume $m_1,\cdots,m_n>0$ such that (H1) and (H2) is in force and take $p>\max\{1,(1-\bar m)N/2)\}$. Then for every $u_0\in L^1(\mathbb{R}^N)\cap L^{p}(\mathbb{R}^N)$, the solution to problem \eqref{APM} satisfies

\begin{equation}\label{Linfty lp}
\|u(t)\|_\infty \leq C t^{-\gamma_p} \|u_0\|^{\delta_p}_p
\end{equation}
with $\gamma_p=(\bar m - 1 + 2p/N)^{-1}$, $\delta_p=2p\gamma_p/N$, the constant C depends on $\bar m,p$ and $N$.
\end{theorem}

\noindent {\sl Proof.} We use a classical parabolic Moser iterative technique. Without lost of generality we can assume that $u$ is smooth. Indeed such assumption can be removed by approximation as in Subsection \ref{ssec.approx}.

\medskip
Case 1: $N>2$.

Let $t>0$ be fixed, and consider the sequence of times $t_k=(1-2^{-k}))t$. As in the proof of Proposition \ref{decay of the $L^p$} we multiply the equation by $|u|^{p_k-2}u$, $p_k\geq p_0>1$, we integrate in $\mathbb{R}^N\times[t_k, t_{k+1}]$. Using Sobolev inequality \eqref{Sobolev sum} with $   r_i  =2$ and $\lambda_i=\frac{m_i+p_k-1}{2}$ and  the decay of the $L^p$ norms given in Proposition \ref{decay of the $L^p$} we get
\begin{equation}
\begin{split}
\|u(t_{k+1})\|_{p_{k+1}}\leq \left[4 C_S
p_k(p_k-1) \min_i \frac{m_i}{(m_i+p_k-1)^2}2^{-(k+1)}t\right]^{-\frac{s}{p_{k+1}}}\|u(t_k)\|_{p_k}^{\frac{sp_k}{p_{k+1}}},
\end{split}
\end{equation}
where $p_k+1=s(p_k+\bar m-1)$ and $s=\frac{N}{N-2}$.

First of all we observe that taking as starting exponent $p_0=p>\max\{1,\frac{N}{2}(1-\bar m)\}1$ it is easy to obtain the value of the sequence of exponents,
$$p_k = A(s^k-1)+p, \quad \text{ with }A=p+(\bar m-1)\frac{N}{2}>0.$$
In particular we get $p_{k+1}>p_k$, with $\lim_{k\rightarrow +\infty} p_k =+\infty$. {Observe that
$$\frac{1}{p_k(p_k-1)\min_i\frac{m_i}{(m_i+p_k-1)^2}}=\frac{(m_j+p_k-1)^2}{p_k(p_k-1)m_j}\leq
\frac{(m_j+p_k-1)^2}{(p_k-1)^2m_j}$$
for some $j$. Moreover,
$$\frac{1}{2C_S}\frac{(m_j+p_k-1)^2}{(p_k-1)^2m_j}\leq \frac{1}{2C_S} \frac{(m_j+p-1)^2}{(p-1)^2m_j}:=c$$}
Now, if we denote $U_k=\|u(t_k)\|_{p_k}$, we have
$$U_{k+1}\leq 2^{\frac{ks}{p_{k+1}}}c^{\frac{s}{p_{k+1}}} t^{-\frac{s}{p_{k+1}}}U_k^{\frac{sp_k}{p_{k+1}}}.$$
This implies
$$U_k\leq 2^{\alpha_k} c^{\beta_k}t^{-\beta_k}U_0^{\delta_k}$$
with the exponents
$$\alpha_k=\frac{1}{p_k}\sum_{j=1}^{N-1}(k-j)s^j\longrightarrow\frac{N(N-2)}{4A},$$
$$\beta_k=\frac{1}{p_k}\sum_{j=1}^Ns^j=\frac{1}{A(s^k-1)+p}\frac{s^k-1}{s-1}s\longrightarrow \frac{s}{A(s-1)}$$
$$\delta_k =\frac{s^kp}{p_k}\longrightarrow\frac{p}{A}.$$
We conclude that
$$\|u(t)\|_\infty=\lim_{k\rightarrow +\infty}U_k\leq Ct^{- \frac{N}{2A}}U_0^{\frac{p}{A}}$$
\textit{i.e.} \eqref{Linfty lp}.

\medskip
Case 2: $N=2$. Starting from Sobolev inequality \eqref{Sobolev pN bis} instead of \eqref{Sobolev sum} with $p_i=2$ and $\lambda_i=\frac{m_i+p-1}{2}$
we get
$$\|u(t)\|_q\leq C t^{-\frac{1}{\left[\bar m-1+p(1-2/q)\right]}} \|u_0\|^{-\frac{p}{\left[\bar m-1+p(1-2/q)\right]}}_p$$
for every $q>2$. We conclude passing to the limit on $q$ as $q\rightarrow+\infty$.
\qed

\medskip

The constant in the previous calculations blows up both as $p\rightarrow1^+$. Nevertheless, an interpolation argument allows to obtain the desired $L^1-L^\infty$ smoothing effect.

\bigskip

\noindent {\sl Proof of Theorem \eqref{L1LI}.}
Putting $\tau_k= 2^{-k}t$, estimate  \eqref{Linfty lp} with (for instance) $p=2$  applied in the interval $[\tau_1,\tau_0]$ gives
$$\|u(t)\|_\infty\leq c(t/2)^{-\gamma_2}\|u(\tau_1)\|_2^{4\gamma_2/N}
\leq
c (t/2)^{-\gamma_2}
\|u(\tau_1)\|_1^{2\gamma_2/N}
\|u(\tau_1)\|_\infty^{2\gamma_2/N}
.$$
We now apply the same estimate in the interval $[\tau_2,\tau_1]$, thus getting
$$\|u(t)\|_\infty\leq c(t/2)^{-\gamma_2}\|u(\tau_1)\|_2^{4\gamma_2/N}
\leq
c (t/2)^{-\gamma_2}
\|u(\tau_1)\|_1^{2\gamma_2/N}
\left(
c (t/4)^{-\gamma_2}
\|u(\tau_1)\|_1^{2\gamma_2/N}
\right)^{2\gamma_2/N}
.$$
Iterating this calculation in $[\tau_k,\tau_{k-1}]$, using Proposition \ref{decay of the $L^p$}, we obtain
$$\|u(t)\|_\infty\leq c^{a_k}2^{b_k}(t)^{-d_k}
\|u(0)\|_1^{e_k}
\|u(\tau_k)\|_2^{f_k}.$$

Recalling that $\bar m>m_c$ (\textit{i.e.} $2\gamma_2/N<1$), we see that the exponents
satisfy, in the limit $k\rightarrow +\infty$,
$$a_k=\sum_{j=0}^{k-1}
\left(\frac{2\gamma_2}{N}\right)^j\longrightarrow
\frac{2\alpha}{N}+1$$
$$b_k=
\sum_{j=0}^{k-1}\gamma_2(j+1)
\left(\frac{2\gamma_2}{N}\right)^j
\longrightarrow
\frac{(\bar m-1)N+4}{[(\bar m-1)N+2]^2},$$
$$d_k=\gamma_2a_k\longrightarrow \alpha,$$
$$e_k=a_k-1\longrightarrow \frac{2\alpha}{N},$$
$$f_k=2\left(\frac{2 \gamma_2}{N}\right)^k\longrightarrow 0.$$
\qed


\medskip

\section*{Acknowledgments}

F. Foe and B. Volzone were partially supported by TAMPA of the Italian INDIUM (National Institute of High Mathematics) and  J. L. V\'azquez was funded by  grants PGC2018-098440-B-I00 and PID2021-127105NB-I00 from MUCIN (the Spanish Government). J.~L.~V\'azquez is an Honorary Professor at Univ. Complutense de Madrid and member of IMG.  B.Volzone would like to thank Assimilator D'Aquinas for interesting discussions about suitable numerical methods, and F\'elix del Teso  for supplying us with accurate numerical computations. Finally,    we wish to thank the referees for suggestions that improved the quality of the paper.

\

\end{document}